\documentclass[a4paper,12pt,oneside,reqno]{amsart}
\pdfoutput=1
\usepackage{amssymb,amsfonts,amsmath,amsthm,amscd}
\usepackage{graphicx, color}
\newcommand{\beq}{\begin{equation}}
\newcommand{\eeq}{\end{equation}}
\newcommand{\bea}{\begin{aligned}}
\newcommand{\eea}{\end{aligned}}
\newcommand{\bdm}{\begin{displaymath}}
\newcommand{\edm}{\end{displaymath}}
\newcommand{\barr}{\begin{array}}
\newcommand{\earr}{\end{array}}
\newcommand{\ben}{\begin{enumerate}}
\newcommand{\een}{\end{enumerate}}
\newcommand{\bde}{\begin{description}}
\newcommand{\ede}{\end{description}}

\numberwithin{equation}{section}
\newtheorem{teor}{Theorem}[section]
\newtheorem{prop}[teor]{Proposition}
\newtheorem{lem}[teor]{Lemma}

\newtheorem{Def}[teor]{Definition}

\newtheorem{rem}[teor]{Remark}

\newcommand{\R}{\mathbb{R}}

\newcommand{\N}{\mathbb{N}}

\newcommand{\PP}{\mathbb{P}}

\newcommand{\al}{\alpha}

\newcommand{\de}{\delta}

\newcommand{\X}{\mathcal{X}}
\newcommand{\M}{\mathcal{M}}
\newcommand{\g}{\mathfrak{g}}

\begin{document}
\title[Trait substitution tree]{Multi-time scales in adaptive dynamics: microscopic interpretation of a trait substitution tree model}
\author {Anton Bovier and Shi-dong Wang}

\address{A. Bovier\\Institut f\"ur Angewandte Mathematik\\Rheinische
   Friedrich-Wilhelms-Uni\-ver\-si\-t\"at Bonn\\Endenicher Allee 60\\ 53115
   Bonn, Germany}
\email{bovier@uni-bonn.de}
\address{S.-D. Wang\\Department of Statistics\\University of Oxford
\\1 South Parks Road
\\ Oxford, OX1 3TG, UK}
\email{shidong.wang@stats.ox.ac.uk}


\thanks{A. Bovier is supported in part by the German Research Foundation in the SFB 611
and the Hausddorff Center for Mathematics. S.-D. Wang was supported by a Hausdorff Scholarship while at the University of Bonn, and EPSRC Grant EP/I01361X/1 while at the University of Oxford.}

\date{\today}

\begin{abstract}
We consider a fitness-structured population model with competition
and migration between nearest neighbors. Under a combination of large
population and rare migration limits we are particularly interested
in the asymptotic behavior of the total population partition on
supporting trait sites. For the population without mutation on a
finite trait space we obtain the equilibrium configuration and
characterize the right time scale for fixation. For the model with
mutation on an infinite trait space a jump process-trait
substitution tree model is established on a rarer mutation time
scale against the rare migration constrained in terms of a large
population limit. Due to a change of the fitness landscape provoked by a new mutant,
some temporarily unfit types can be recovered from time to time.
In the end we shed light to illustrate sexual reproduction
in a diploid population on the genetic level.
\end{abstract}

\maketitle

\tableofcontents


\section{Introduction}\label{section 3.1}
In recent years a spatially structured population with migration
(dispersion) and local regulation, proposed by Bolker and Pacala
\cite{BP97}, Dieckmann and Law \cite{LD02} (in short BPDL process),
has attracted particular interest both from biologists and
mathematicians. It has several advantages over general branching
processes, which make it more natural as population models: the
quadratic competition term is used to prevent the population size
from escaping to infinity and the migration term is used to
transport the population mass from one colony to unoccupied colonies
for survival, and to further get colonized. There are mainly two
highlights of related papers. For instance, Etheridge \cite{Eth04}, Fournier and
M\'el\'eard \cite{FM04}, Hutzenthaler and Wakolbinger \cite{HW07},
have studied the extinction and
survival problems. Champagnat \cite{Cha06}, Champagnat and Lambert
\cite{CL07}, Champagnat and M\'el\'eard \cite{CM10}, M\'el\'eard and Tran \cite{MT09}, 
Dawson and Greven \cite{DG10} focus more on its long time behavior by multi-scale
analysis methods.

The main ingredient behind this model is logistic branching random
walks, that is, a combination of logistic branching populations with
spatial random walks (or migration) on trait sites. Under a
combination of a large population and rare mutation limits, a
so-called trait substitution sequence model (in short TSS) is
derived in \cite{Cha06}. The heuristics leading to the TSS model is
based on the biological assumptions of large population and rare
mutation, and on another assumption that no two different types of
individuals can coexist on a long time scale: the selective
competition eliminates one of them. On the one hand, coexistence and
diversity after entering of new mutants are not allowed due to the
deficient spatial structure. On the other hand, natural selection is
not only limited to competition mechanism but also is often combined
with a survival strategy-migration mechanism. In spite of this
heuristics, this model is still lack of a rigorous mathematical
basis.

The adaptive-dynamics approach is controversially debated since it
was criticized only feasible in the context of phenotypic approach.
However, the link with its corresponding genetic insight has rarely been
developed (see Eshel \cite{Esh96}). As far as sexual reproduction
is concerned, population genetic models have dominated for
many years since they have been proved powerful to model diploid populations on the genetic level. For a
finite gene pool of fixed size, main evolutionary mechanisms like
mutation, selection, and reproduction are theoretically
tractable though they can take a role in a very complicated way especially after
sexual reproduction gets involved. The effect of
sexual reproduction is more complicated to characterize mostly because random
shuffling of genes may create many genotypes for natural selection to act on,
 which makes mathematical analysis more difficult (e.g., see \cite{CMM11} for the case with three genotypes combined by two alleles). 
In contrast, adaptive-dynamics approach is mainly concerned
with the long-term evolutionary property but usually ignores the
genetic complications. Is there a way to embody features such
as sexual reproduction arising on
the genetic level but at the same time in which one can study its long term behavior via
the quantitative trait method, i.e., taking advantage of
adaptive-dynamics approach on the phenotypic level? This is the
biological motivation of this paper. 

In this paper we propose a new model to justify the above arguments. We introduce
a spatial migration mechanism among possible genotypes, which can be
viewed as a result of fusion of any pair of alleles out of a fixed finite allele pool
of a diploid population. After natural selection acting on a short-term
evolution time scale, the population can attain an equilibrium
configuration according to the fitness landscape. Each time there
enters a new mutant gene (allele) into the gene pool, the genotype
space is enlarged due to formation of new genotypes, and the spatial migration
can be used to characterize the reshuffling procedure on the way to
a new equilibrium configuration. Loosely speaking, the spatial
movement is used to compensate the simplicity of genetic
reproduction in adaptive-dynamics approach. The critical point we
need to take care of is to distinguish these different time scales
after introducing fitness spatial structure in the model.

The novelty of this model differs from previous models in three key
aspects. Firstly, no genetic information is lost on any time scale.
Some genotype containing a specific deleterious allele may be
invisible due to its temporarily low fitness on the migration time
scale, but it can recover on a longer mutation time scale due to the entering of a new
mutant allele and the reshuffling of the genotype space. For example, some
epidemic virus may become popular periodically because of a
change of its mutated genetic structure or a genetic change of its
potential carrier. Secondly, thanks to the fitness spatial structure
endowed on finite genotype space, coexistence is allowed under the
assumption of nearest-neighbor competition and migration. This
distinguishes our model from the classical adaptive-dynamics,
which often converges to a monomorphic equilibrium. What is more, we
derive a well-defined branching tree structure in the limiting
system, which is like a spatial version of the Galton-Watson
branching process. Last but not least, the idea of introducing the
spatial migration to interpret sexual reproduction can provide a link
between adaptive-dynamics and its genetics counterpart. In
particular, similar consideration can be done to quantify more
complicated sexual reproduction model than our toy case from the genetics
side by mapping it to spatial migration model (see Section \ref{section 3.5}).

As a reminder, we want to mention some recent progress in the interacting fields of adaptive dynamics and genetics. 
Champagnat and M\'el\'eard  \cite{CM10} relax the assumption of non-coexistence condition in \cite{Cha06}
and obtain a polymorphic evolution sequence (PES) as a generalization of the TSS model,
allowing coexistence of several traits in the population. However, still unfit allelic traits can be excluded from the 
evolutionary history and may never recover, depending on the Jacobian matrices of Lotka-Volterra systems.
Recently, Collet, M\'el\'eard and Metz \cite{CMM11} consider a diploid population model with sexual reproduction,
and obtain that population behaves on the mutation time scale as a jump process moving
between homozygous states (genotypes comprising of a pair of identical alleles).
Although their model puts a rigorous basis on Mendelian diploids, as mentioned in \cite[Section 6]{CMM11},
it is still under the restriction of an unstructured population and single locus genetics. 
Evans et al. \cite{CE09, ESW07} study a continuous time
evolving distribution of genotypes called \emph{mutation-selection balance} model where
recombination acts on a faster time scale than mutation and
selection. The intuition behind their asymptotic result is that the
mutation preserves the Poisson property whereas selection and
recombination respectively drive the population distribution away
from and toward Poisson. If all three processes are operating
together, one expects that recombination mechanism disappears in the limiting system.
This in some sense motivates us to specify migration in adaptive-dynamics to express one kind of 
genetic reshuffling like recombination in Evan's model. 
And migration should act on some well-defined fitness structure.

In this paper we are interested in the case when the migrant event
is rare with respect to branching events but not that rare as in
\cite{Cha06} (see Figure \ref{Figure_ParameterRegion}). In contrast,
we assume that there are infinite migrants from a resident
population on the natural time scale. Let a parameter $\epsilon$ be
the migration rate and $K$ be proportional to the initial population
size. We will impose the rare migration constraint $1\ll
K\epsilon\ll K$ on the population (see parameter region II in Figure
\ref{Figure_ParameterRegion}). As far as a finite-trait dynamic
system is concerned, to find out the exact fixation time scale
expressed in terms of the migration rate and population size is of
particular interest for us. Since the original model is not easy for
us to study due to the complicated interactions, we present here a
slightly modified model of the one in \cite{Cha06} but retaining the
essential machinery founded in the original model. This paper is
restricted with nearest-neighbor competitions and migrations along
the monotone fitness landscape. What is more, in order to study the
long time behavior, we introduce mutations to drive the population
to move forwards to more fitter configuration on a rare mutation time
scale, which is longer than the fixation time scale. Note that the
limit theorem arising in \cite{Cha06} can be applied consistently in
the model developed in this paper.

The purpose of this paper and the accompanying one \cite{BW11a} is
to justify a trait substitution tree process (in short TST) to
illustrate the coexistence phenomenon with spatial structure in
evolution theory, which is a purely atomic finite measure-valued
process. The present one is derived from the microscopic point of
view while the other one \cite{BW11a} is from the macroscopic
point of view. Combining these two papers together with \cite{Cha06, CM10},
the entire framework on (rare) migration against (large) population
limit can be fully characterized, and it results in different
rescaling limits, TSS and TST respectively on different time scales.
In summary, the entire framework is as follows:
\begin{itemize}
\item Take large population and rare migration simultaneously by $K\epsilon\ll\frac{1}{\ln K}$, it leads to a TSS limit in \cite{Cha06}.
\item Firstly let $K\to\infty$, then add rare mutation by $\ln\frac{1}{\epsilon}\ll\frac{1}{\sigma}$ as $\epsilon\to 0$, it leads to a TST limit in \cite{BW11a}.
\item Take large population, rare migration and even rarer mutation all simultaneously constrained by $1\ll K\epsilon\ll K,\,\ln\frac{1}{\epsilon}\ll\frac{1}{K\sigma}$. That is our goal in this paper.
\end{itemize}

 \begin{figure}[hbtp]
 \centering
 \def\svgwidth{300pt}
 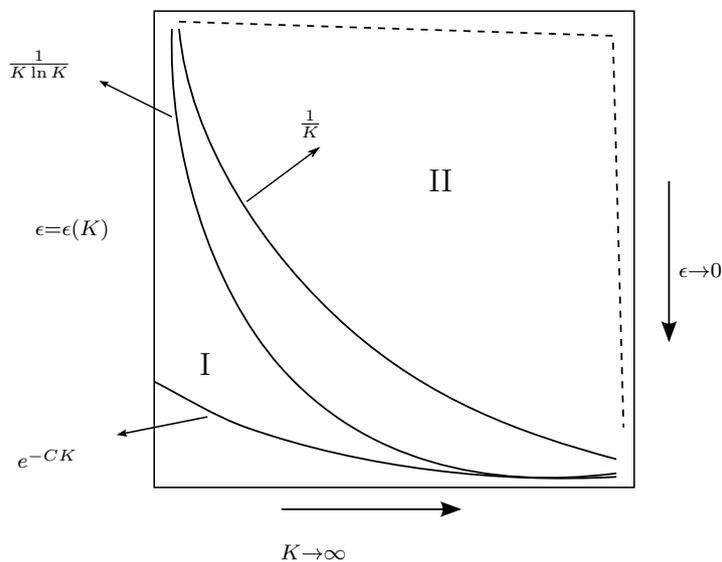
 \caption{{\footnotesize Division of parameter region: migration rate $\epsilon$ against population size $K$}}.
 \label{Figure_ParameterRegion}
 \end{figure}

The remainder of the paper is structured as follows. In Section
\ref{section 3.2}, we present a description of the individual-based
model. In Section \ref{section 3.3}, we consider the case without
mutation but on a finite trait space, and characterize the rare
migration limit against the large population limit. In Section
\ref{section 3.4}, concerning a modified population supported on an
infinite trait space by introducing mutations, we justify a
so-called trait substitution tree processes in the rare mutation
limit, which also appeared in \cite{BW11a}. In section
\ref{section 3.5}, we apply the previous results to a diploid
population. In the last section, related proofs for results in
previous sections are provided.


\section{Microscopic model}\label{section 3.2}

We begin with a description of an individual-based model. Assume that the population at time $t$ is composed of a finite number $I_t$ individuals characterized by their phenotypic traits $x_1(t), \ldots, x_{I_t}(t)$ belonging a compact subset $\X$ of $\R^d$.
We denote by $\M_F(\X)$ the set of non-negative finite measures on $\X$. Let $\M_a(\X)\subset \M_F(\X)$ be the set of counting measures on $\X$:
\[
\M_a(\X)=\left\{\sum\limits_{i=1}^n\de_{x_i}: x_1,\ldots, x_n\in\X, n\in\N\right\}.
\]
Then, the population process at time $t$ can be represented as:
\[
\nu_t=\sum\limits_{i=1}^{I_t}\delta_{X_i(t)}.
\]
Let $B(\X)$ denote the totality of functions on $\X$ which are bounded and measurable. For any $f\in B(\X)$, $\nu\in\M_F(\X)$, we use notation $\langle\nu,f\rangle=\int f d\nu$.

Let's specify the population process $(\nu_t)_{t>0}$ by introducing a sequence of demographic parameters:
\begin{itemize}
  \item $b(x)$ is the birth rate from an individual with trait $x$.
  \item $d(x)$ is the death rate of an individual with trait $x$ because of ``aging''.
  \item $\al(x,y)$ is the competition kernel felt by some individual with trait $x$ from another individual with trait $y$.
  \item $m(x, dy)$ is the migration law of an individual from trait site $x$ to site $y$.
  \item $\mu(x)$ is the mutation rate of an individual with trait $x$.
  \item $p(x,dh)$ is the law of mutant variation $h=y-x$ between a mutant $y$ and its resident trait $x$. Since the mutant trait $y=x+h$ should belong to $\X$, this law has its support in $\X-x:=\{y-x: y\in\X\}\subset \R^d$.
\end{itemize}

To specify the model without mutation mechanism, the infinitesimal generator of the $\M_a(\X)$-valued process is
given as follows, for any $F\in B(\M_a(\X))$:
\beq
\bea
LF(\nu)&=\sum\limits_{i=1}^I\left[F(\nu+\delta_{x_i})-F(\nu)\right]b(x_i)\\
          &+\sum\limits_{i=1}^I\left[F(\nu-\delta_{x_i})-F(\nu)\right]\left(d(x_i)+\sum\limits_{j\neq i}^I\alpha(x_i,x_j)\right)\\
          &+\sum\limits_{i=1}^I\sum\limits_{x_j\neq\,x_i}\left[F(\nu+\delta_{x_j}-\delta_{x_i})-F(\nu)\right]m(x_i,x_j).
\eea
\eeq
The first term above describes the clonal reproduction at the mother's site. The second term describes death of an individual $x_i$ either due to aging or competition from another individual $x_j$. And the last term describes the migration of an individual from a trait site $x_i$ to a site $x_j$.

By introducing a parameter $K\in\N$, we rescale the population size and competition kernel by $K$. We will show later, as $K$ tends to infinity, one can get different large population limits by various well-chosen rescaling procedures. Furthermore, the population process can be parameterized by another parameter $\epsilon$ governing the rate of migration law $m(x_i,x_j)$ in terms of population size scaling parameter $K$.

For any $K\in\N$, instead of studying the above process $(\nu_t^K)_{t\geq0}$, it is more convenient to consider a sequence of rescaled measure-valued processes:
\begin{equation}
X_t^K:=\frac{1}{K}\nu_t^K=\frac{1}{K}\sum_{i=1}^{I_t^K}\delta_{x_i}
\end{equation}
where $X_{\cdot}^K$ is a $\M_F(\X)-$valued process with the following infinitesimal generator:
\beq
\bea\label{generator_X_epsilon}
 L^K F(\nu)
    &= \int_{\X}\left[F(\nu+\frac{\delta_x}{K})-F(\nu)\right]b(x)K\nu(dx)\\
    &+\int_{\X}\left[F(\nu-\frac{\delta_x}{K})-F(\nu)\right]\left(d(x)+\int_{\X}\alpha(x,y)\nu(dy)\right){K}\nu(dx)\\
    &+\epsilon\int_{\X}\int_{\X}\left[F(\nu+\frac{\de_{y}}{K}-\frac{\de_{x}}{K})-F(\nu)\right]m(x,dy)K\nu(dx).
\eea
\eeq

Notice that we rescale the competition kernel $\al$ by $K$ so that the system mathematically makes sense when we take a large population limit. From the biological point of view, $K$ can be interpreted as scaling the resource or area available.

Let us denote by ($\mathbf{A}$) the following assumptions.
\begin{description}
\item[(A1)]$\exists ~\bar{b},~ \bar{d},~\underline{\alpha}~, \bar{\alpha},$ such that $~0<b(x)\leq\bar{b},~ 0<d(x)\leq\bar{d},~ 0<\underline{\alpha}\leq\alpha(x,y)\leq\bar{\alpha}, \text{and}~ b(x)-d(x)>0, ~\forall x\in\X.$
\item[(A2)]$\forall ~x,y$ $\in$ $\X$,~$\bar f(x,y)\cdot \bar f(y,x)<0$, where the fitness functions

 $\bar f(y,x)=b(y)-d(y)-\alpha(y,x)\bar{n}(x)$ and $\bar{n}(x)=\frac{b(x)-d(x)}{\alpha(x,x)}$,

 $\bar f(x,y)=b(x)-d(x)-\alpha(x,y)\bar{n}(y)$ and $\bar{n}(y)=\frac{b(y)-d(y)}{\alpha(y,y)}$.
\end{description}

Notice that assumption (A1) guarantees that the process with the infinitesimal generator \eqref{generator_X_epsilon} is well defined (refer to \cite{FM04}). Assumptions (A2) gives the non-coexistence condition for any pair of distinct competing traits.


\section{Early time window on an finite trait space as $K\to\infty$}\label{section 3.3}

We firstly review some exsiting results for this model. Champagnat \cite[Theorem 1]{Cha06} proved the following result by the time scales separation technique, which can be extended to a more general case in accelerated population dynamics \cite{Wang11}.

\begin{teor}\label{TSS Theorem of Champagnat}
Admit assumptions $(\textbf{A}1)$ and $(\textbf{A}2)$. Suppose that $X_0^K=\frac{N_0^K}{K} \delta_x$ such that $\frac{N_0^K}{K}\stackrel{\text{law}}{\to} n_0
>0$ as $K\rightarrow +\infty$, and $\forall ~C>0$,
\beq\label{Champagnat condition}
\exp\{-CK\}\ll K\epsilon \ll \frac{1}{\ln K}.
\eeq
Then, $(X^K_{t/K\epsilon}, t\geq 0)$ converges in the sense of f.d.d. to
\begin{align*}
Y_t= \left\{
  \begin{array}{ll}
   n_0\delta_x, & t=0 \\
   \bar{n}(\eta_t)\delta_{\eta_t}, & t>0
 \end{array}
\right.
\end{align*}
where the Markov jump process $(\eta_t, t\geq 0)$ satisfies $\eta_0=x$ with an infinitesimal generator:
\begin{equation}
          A\varphi(x)=\int_{\X}(\varphi(y)-\varphi(x))
          \bar{n}(x)\frac{[\bar f(y,x)]_+}{b(y)}m(x,dy).
\end{equation}
\end{teor}

 \begin{figure}[hbtp]
 \centering
\includegraphics[width=220pt]{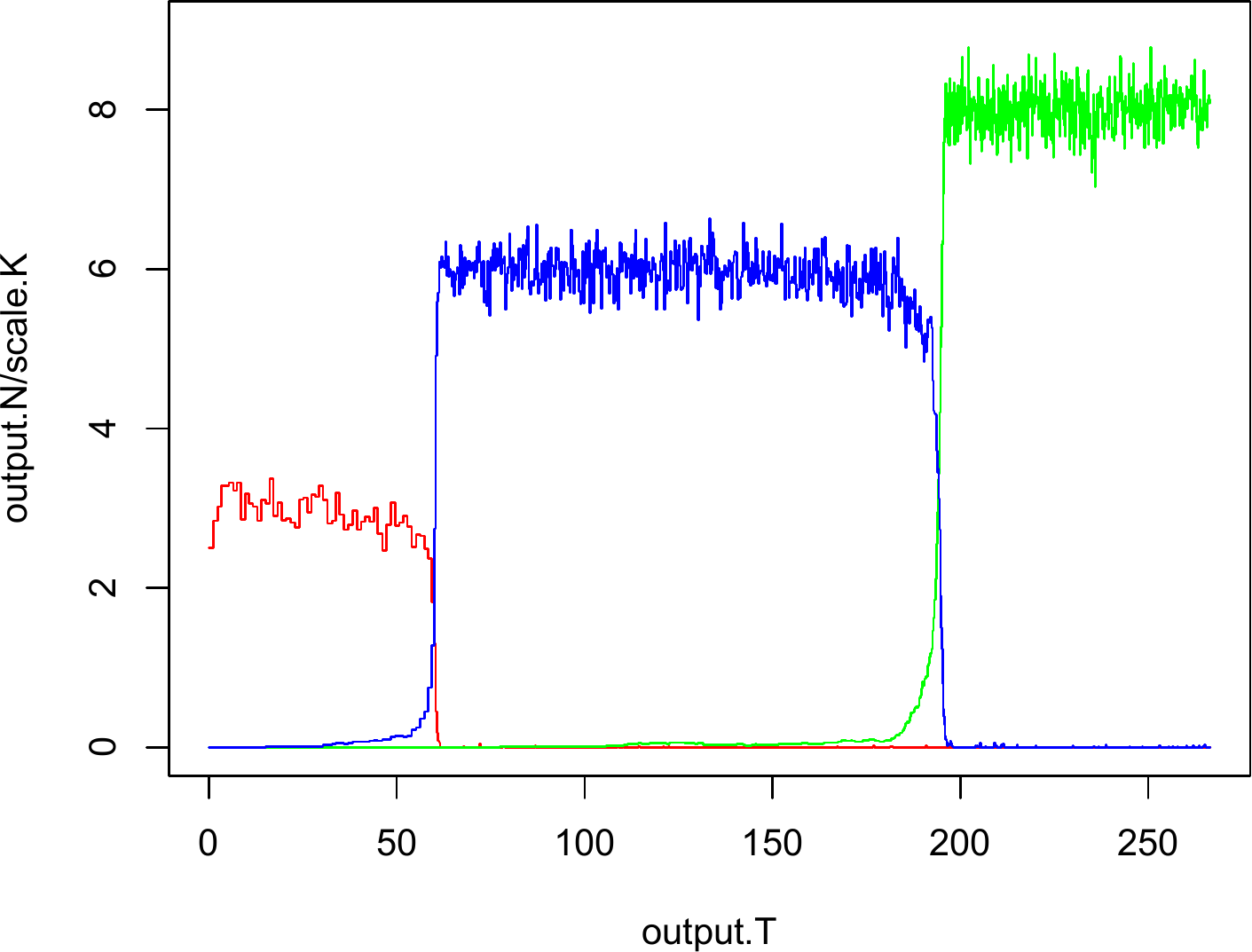}
\includegraphics[width=220pt]{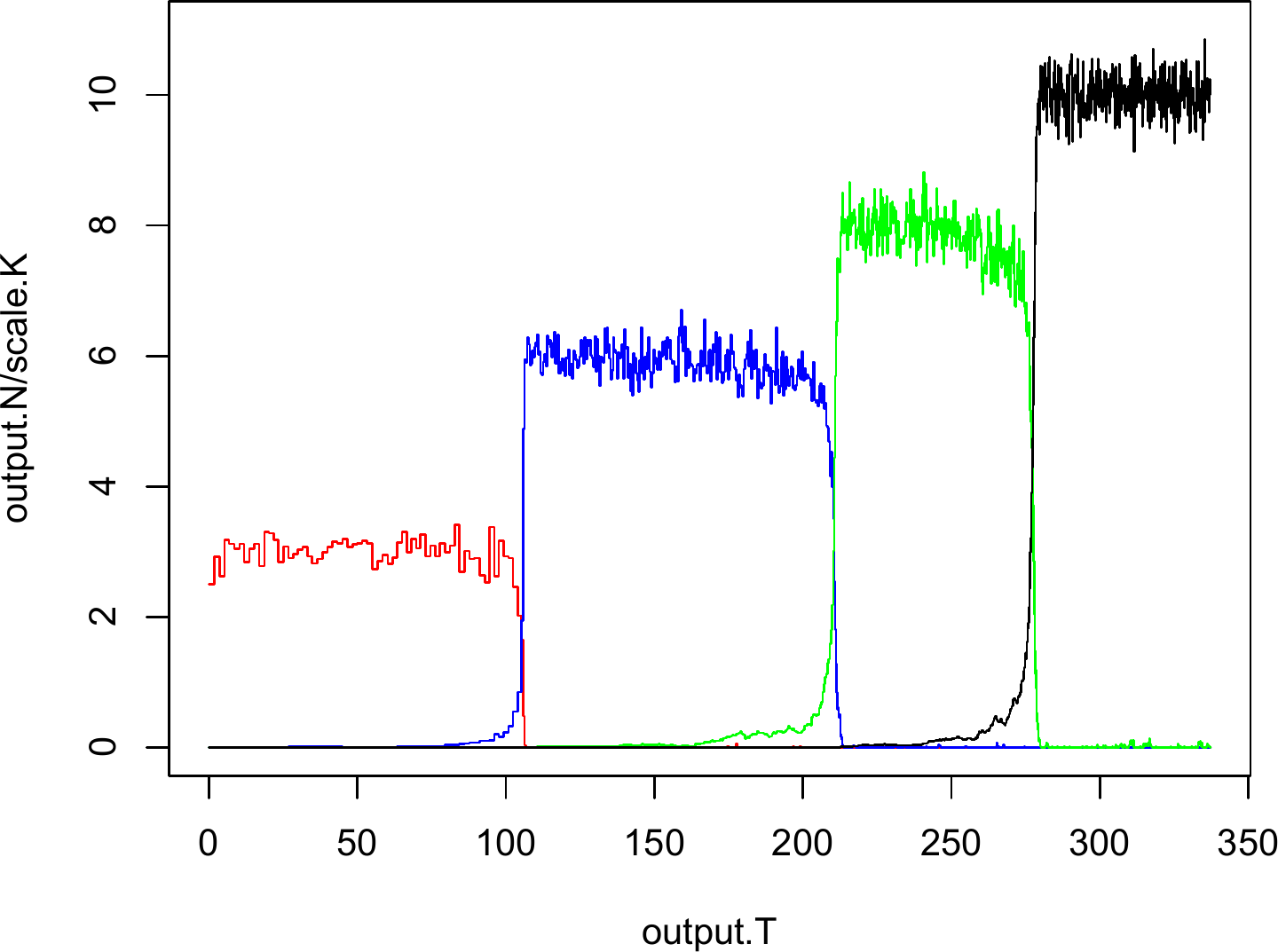}
\caption{\footnotesize Simulations of the trait substitution sequence model arising in Theorem \ref{TSS Theorem of Champagnat}.}
\label{Figure TSS_simulation}
 \end{figure}

\begin{rem}
\begin{itemize}
\item
The migration time scale is of order $\frac{1}{K\epsilon}$ whereas the fixation time scale starting from one migrant is of order $\ln K$. The population is kept monomorphic on the rare migration time scale. The rare migration parameter region constrained by \eqref{Champagnat condition} is denoted by the region $\textbf{I}$ in Figure \ref{Figure_ParameterRegion}.
\item
As showed in Figure \ref{Figure TSS_simulation}, it simulates a TSS model with trait space comprising of three types in the left panel while it simulates a four-type case  in the right panel. We mark the population density of trait $x_0, x_1, x_2, x_3$ by red, blue, green and black colored curves respectively. Take $b(x_0)=3,\, b(x_1)=6,\, b(x_2)=8,\, b(x_3)=10$ and death rates $d(x_i)\equiv 0,\,i=0,1,2,3.$ Take competition kernel $\al\equiv 1$,~ migration kernel $m\equiv 0.5$, and migration parameter $\epsilon=K^{-2}$, where initial population size $K=100$.
\end{itemize}
\end{rem}

In \cite{BW11a}, we firstly let $K$ tend to infinity in \eqref{generator_X_epsilon} and obtain a deterministic limit. Then, we consider the rescaling limit of the deterministic system supported on a finite trait space in a slow migration limit. That is actually an extreme case where it attains the so-called \emph{trait substitution tree} by taking a two-step limit along the marginal path (see dashed path in Figure \ref{Figure_ParameterRegion}). In terms of the individual-based population, it is of particular interest for us to give a microscopic interpretation of the TST process under apropriate constraints. Prior to the following theorem, we list the following assumption (\textbf{B}) to attain our main result later.
\begin{description}
\item [(B1)] For any finite number of types $L\in\N$, it has a monotonously increasing fitness landscape: $x_0\prec x_1\prec\ldots\prec x_L$, where $x_0\prec x_1$ denotes $$\bar f(x_0,x_1)<0,\,\bar f(x_1,x_0)>0$$.
\item [(B2)] Nearest-neighbor migration and competition, i.e.
    $m(x_i,x_j)=\al(x_i,x_j)=0$ for any $\mid i-j\mid>1$.
\item [(B3)] For any $i\geq 2$,
    \beq
    \frac{i}{b(x_i)-d(x_i)}\geq \frac{1}{\bar{f}(x_i,x_{i-1})}+\frac{1}{\bar{f}(x_{i-1},x_{i-2})}+\cdots+\frac{1}{\bar{f}(x_1,x_0)}.
    \eeq
\end{description}

Note that assumption (B3) is not necessary for us to obtain the following theorem. There actually exist a variety of different possible paths to converge to the equilibrium configuration determined up to the ordered sequence of traits as in assumption (B1).
However, thanks to assumption (B3), it brings us a lot convenience to prove the theorem without losing intrinsic features.

We inherit some notations from \cite{BW11a}, denote configurations by $\Gamma^{(L)}:=\sum\limits_{i=0}^l\bar{n}(x_{2i})\de_{x_{2i}}$ if $L=2l$ and $\Gamma^{(L)}:=\sum\limits_{i=1}^{l+1}\bar{n}(x_{2i-1})\de_{x_{2i-1}}$ if $L=2l+1$ for any $l\in\N\cup 0$.

\begin{teor}\label{TST Theorem on finite trait space_Micro}
Admit assumptions $(\textbf{A}1)$ and $\textbf{B}$. Consider the processes $(X_t^K)_{t\geq 0}$ on the trait space $\X=\{x_0,x_1,\ldots,x_L\}$. Suppose that $X_0^K=\frac{N_0^K}{K} \delta_{x_0}$ such that $\frac{N_0^K}{K}\stackrel{\text{law}}{\to} n_0
>0$ as $K\rightarrow +\infty$, and
\beq\label{rare migration condition}
1\ll K\epsilon\ll K.
\eeq
Then there exists a constant $\bar t_L>0$, such that for any $t>\bar t_L$,
$\lim\limits_{K\to\infty}X^K_{t\ln\frac{1}{\epsilon}}\stackrel{\text{(d)}}{=}\Gamma^{(L)}$ under the total variation norm.
\end{teor}

 \begin{figure}[hbtp]
 \centering
\includegraphics[width=220pt]{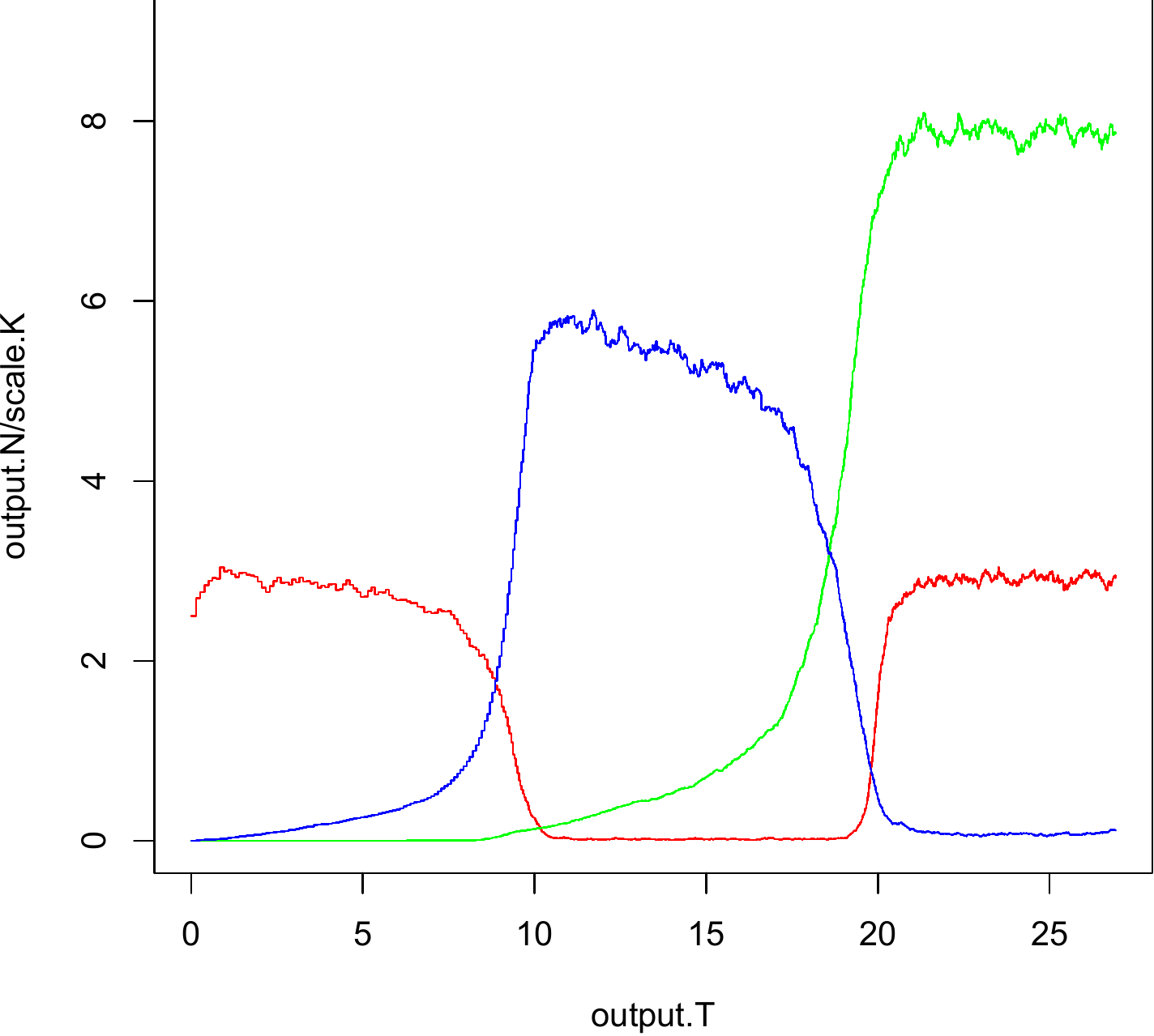}
\includegraphics[width=220pt]{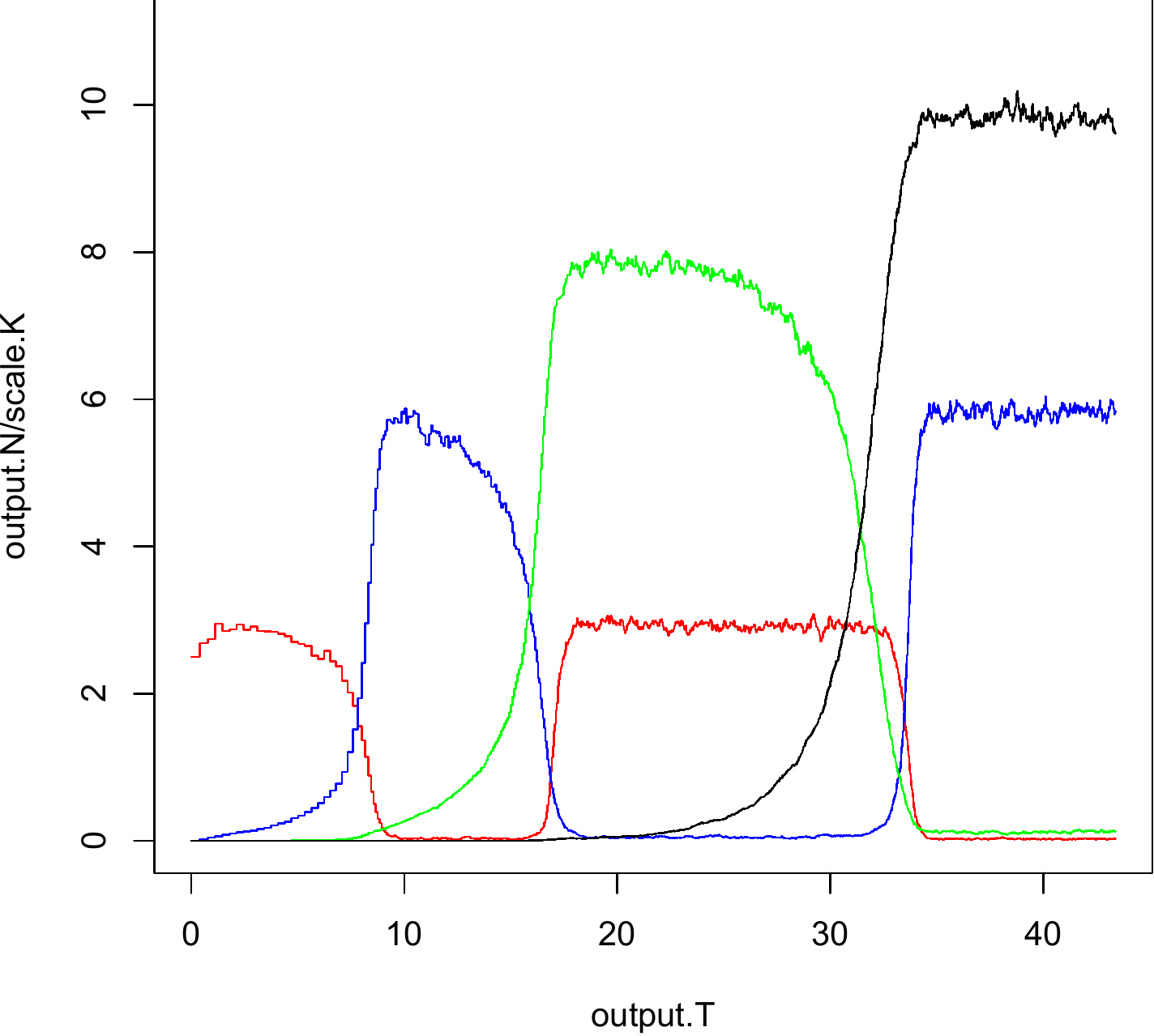}
\caption{\footnotesize Simulations of a trait substitution tree model arising in Theorem \ref{TST Theorem on finite trait space_Micro} on a three- and four-type trait space. }
\label{Figure TST_simulation_Micro}
 \end{figure}

\begin{rem}
\begin{itemize}
\item
We illustrate the theorem by simulations (see Figure \ref{Figure TST_simulation_Micro}). We take all the same parameters as in Figure \ref{Figure TSS_simulation} except replacing $\epsilon=K^{-\frac{4}{5}}$ in the three-type case and $\epsilon=K^{-\frac{3}{4}}$ in the four-type case with initial population size $K=1000$. Obviously, they both satisfy conditions \eqref{rare migration condition}.
\item
The parameter region for rare migration constrained by \eqref{rare migration condition} is denoted by the upper right region \textrm{\textbf{II}} in Figure \ref{Figure_ParameterRegion}. As analyzed in Theorem \ref{TST Theorem on finite trait space_Micro}, the fixation time scale is of order $\ln\frac{1}{\epsilon}$. The stable configuration for the three-type case is $\Gamma^{(2)}=3\de_{x_0}+8\de_{x_2}$ and it is $\Gamma^{(3)}=6\de_{x_1}+10\de_{x_3}$ for the four-type case. We will show in Theorem \ref{TST_Theorem_infinite trait_Micro} that the TST process jumps from $\Gamma^{(2)}$ to $\Gamma^{(3)}$ on an even rarer mutation time scale of order $\frac{1}{K\sigma}$ (see Figure \ref{Figure TST_simulation_Micro with mutation}).
\end{itemize}
\end{rem}


\section{Late time window with mutations as $K\to\infty$}\label{section 3.4}

Following the procedures we build up in \cite{BW11a}, in order to study the evolutionary behavior driven by new input on a much longer time scale, we introduce a mutation mechanism into the population generated by \eqref{generator_X_epsilon}. We now study the model with mutations formulated by the following generator supported on a compact set $\X$:
\beq\label{generator_X_epsilon,sigma_Micro}
\bea
 L^{K,\epsilon,\sigma} F(\nu)
    &= \int_{\X}\left[F(\nu+\frac{\delta_x}{K})-F(\nu)\right]b(x)K\nu(dx)\\
    &+\int_{\X}\left[F(\nu-\frac{\delta_x}{K})-F(\nu)\right]\left(d(x)+\int_{\X}\alpha(x,y)\nu(dy)\right){K}\nu(dx)\\
    &+\epsilon\int_{\X}\int_{\X}\left[F(\nu+\frac{\de_{y}}{K}-\frac{\de_{x}}{K})-F(\nu)\right]m(x,dy)1_{\{y\in
     \textrm{supp}\{\nu\}\}}K\nu(dx)\\
    &+\sigma\int_{\X}\int_{\R^d}\left[F(\nu+\frac{\de_{x+h}}{K})-F(\nu)\right]\mu(x)p(x,dh)K\nu(dx).
\eea
\eeq

Here we denote the process by $X_{\cdot}^{K,\epsilon,\sigma}$ with one more superscript $\sigma$, distinguishing from the one without mutation in the previous section.

Notice that the mutation kernel $p(x,dh)$ is used to introduce a new trait site to the previous finite trait space and enlarge the supporting trait space by one each time there enters a mutant, whereas the migration kernel only acts on current support sites of the population. Later we will see, under some rare mutation constraint (with respect to migration rate), the dominating power for fixation is mainly from exponential growth of intial migration particles. Before proceeding towards the main theorem, we now give some assumptions and the definition of the trait substitution tree, which appeared in \cite{BW11a}.

Assumption (\textbf{C}).
\begin{description}
\item[(C1)] For any given distinct traits $\{x_0, x_1,\cdots, x_n\}\subset\X, n\in\N$, there exists a total order permutation
\beq
x_{n_0}\prec x_{n_1}\prec\cdots\prec x_{n_{n-1}}\prec x_{n_n},
\eeq
where $x\prec y$ means that the fitness functions satisfy $\bar f(x,y)=b(x)-d(x)-\al(x,y)\bar n(y)<0$, and $\bar f(y,x)=b(y)-d(y)-\al(y,x)\bar n(x)>0$.

For simplicity of notation, we always assume $x_0^{(n)}\prec x_1^{(n)}\prec\cdots\prec x_n^{(n)}$ with $x^{(n)}_i=x_{n_i}$ for $0\leq i\leq n$. By adding a new trait $x$ whose fitness is between  $x_j^{(n)}$ and $x_{j+1}^{(n)}$ for some $0\leq j\leq n$, we relabel new traits as following
\beq
x^{(n+1)}_0\prec x^{(n+1)}_1\prec\cdots\prec x^{(n+1)}_n\prec x^{(n+1)}_{n+1},
\eeq
where $x^{(n+1)}_i=x^{(n)}_i$ for $0\leq i\leq j$, $x^{(n+1)}_{j+1}=x$ and $x^{(n+1)}_{i}=x^{(n)}_{i-1}$ for $j+2\leq i\leq n+1$.

\item[(C2)] Competition and migration only occurs between nearest neighbors, i.e. for totally ordered traits in (C1), we have $m(x_i^{(n)},x_j^{(n)})=\al(x_i^{(n)},x_j^{(n)})\equiv 0$ for $\mid i-j\mid>1$.
\end{description}

Under above assumptions we can rewrite the generator \eqref{generator_X_epsilon,sigma_Micro} as following
\beq\label{generator_X_epsilon,sigma_Micro_nearest neighor}
\bea
 L^{K,\epsilon,\sigma} F(\nu)
    &= \int_{\X}\left[F(\nu+\frac{\delta_x}{K})-F(\nu)\right]b(x)K\nu(dx)\\
    &+\int_{\X}\left[F(\nu-\frac{\delta_x}{K})-F(\nu)\right]\left(d(x)+\int_{\X}\alpha(x,y)1_{\{x^-,\,x,\,
     x^+\}}\nu(dy)\right){K}\nu(dx)\\
    &+\epsilon\int_{\X}\int_{\X}\left[F(\nu+\frac{\de_{y}}{K}-\frac{\de_{x}}{K})-F(\nu)\right]1_{\{x^-,\,
     x^+\}}m(x,dy)K\nu(dx)\\
    &+\sigma\int_{\X}\int_{\R^d}\left[F(\nu+\frac{\de_{x+h}}{K})-F(\nu)\right]\mu(x)p(x,dh)K\nu(dx)
\eea
\eeq
where $x^-$ and $x^+$, specified by the total order relation in assumption (C1), are elements in supp$\{\nu\}\subset\X$ satisfying
\[
x^-=\sup \{y\in \textrm{supp}\{\nu\}: \bar f(y,x)<0\}
\]
and
\[
x^+=\inf \{y\in \textrm{supp}\{\nu\}: \bar f(y,x)>0\}.
\]

On the migration time scale, there are a variety of different paths to approach the equilibrium configuration by specifying different coefficients. However, the equilibrium configuration of a finite trait system is always the same up to the ordered sequence determined as in assumption (C1), and the time scale for convergence is always of order $O(\ln\frac{1}{\epsilon})$ as showed in Theorem \ref{TST Theorem on finite trait space_Micro}.

\begin{Def}\label{TST_Definition_infinite tree_micro} A Markov jump process $\{\Gamma_t: t\geq 0\}$ characterized as following is called a trait substitution tree (in short TST) with the ancestor $\Gamma_0=\bar n({x_0})\delta_{x_0}$.
\begin{description}
\item[(i)] For any nonnegative integer $l$, it jumps
    from $\Gamma^{(2l)}:=\sum_{i=0}^l\bar n(x^{(2l)}_{2i})\delta_{x^{(2l)}_{2i}}$
    to $\Gamma^{(2l+1)}$\\
    with transition rate $\bar n(x^{(2l)}_{2k})\mu(x^{(2l)}_{2k})p(x^{(2l)}_{2k},dh)$ for any $0\leq k\leq l$, where
    \begin{itemize}
    \item $\Gamma^{(2l+1)}=\sum_{i=1}^j\bar n(x^{(2l)}_{2i-1})\delta_{x^{(2l)}_{2i-1}}+\bar n(x^{(2l)}_{2k}+h)\delta_{x^{(2l)}_{2k}+h}+\sum_{i=j+1}^l\bar n(x^{(2l)}_{2i})\delta_{x^{(2l)}_{2i}}$\\
        \vspace{4mm}
        if there exists $0\leq j\leq l$ s.t.
        $x^{(2l)}_{2j}\prec x^{(2l)}_{2k}+h\prec x^{(2l)}_{2j+1}$,\\
    \item $\Gamma^{(2l+1)}=\sum_{i=1}^j\bar n(x^{(2l)}_{2i-1})\delta_{x^{(2l)}_{2i-1}}+\sum_{i=j}^l\bar n(x^{(2l)}_{2i})\delta_{x^{(2l)}_{2i}}$\\ \vspace{4mm}
        if there exists $0\leq j\leq l$ s.t.
        $x^{(2l)}_{2j-1}\prec x^{(2l)}_{2k}+h\prec x^{(2l)}_{2j}$.
    \end{itemize}
    Then, we relabel the traits according to the total order relation as in (C1):
    \beq
    x_0^{(2l+1)}\prec x_1^{(2l+1)}\prec\cdots\prec  x_{2l}^{(2l+1)}\prec x_{2l+1}^{(2l+1)},
    \eeq
    where in associate with the first case
    \begin{align*}
    &x_i^{(2l+1)}:=x_i^{(2l)} ~\textrm{for} ~0\leq i\leq 2j,\qquad x^{(2l+1)}_{2j+1}:=x^{(2l)}_{2k}+h,\\
    &x_i^{(2l+1)}:= x_{i-1}^{(2l)}~\textrm{for} ~2j+2\leq i\leq 2l+1,
    \end{align*}
    and in associate with the second case
    \begin{align*}
    &x_i^{(2l+1)}:=x_i^{(2l)} ~\textrm{for} ~0\leq i\leq 2j-1,\quad x^{(2l+1)}_{2j}:=x^{(2l)}_{2k}+h,\\
    &x_i^{(2l+1)}:= x_{i-1}^{(2l)} ~\textrm{for} ~2j+1\leq i\leq 2l+1.
    \end{align*}

\item[(ii)] For any nonnegative integer $l$, it jumps
    from $\Gamma^{(2l+1)}:=\sum_{i=1}^{l+1}\bar n(x^{(2l+1)}_{2i-1})\delta_{x^{(2l+1)}_{2i-1}}$ to $\Gamma^{(2l+2)}$
    \vspace{4mm}
    with transition rate $\bar n(x^{(2l+1)}_{2k-1})\mu(x^{(2l+1)}_{2k-1})p(x^{(2l+1)}_{2k-1},dh)$ for any $1\leq k\leq l+1$, where
    \vspace{4mm}
    \begin{itemize}
    \item $\Gamma^{(2l+2)}=\sum_{i=1}^j\bar n(x^{(2l+1)}_{2(i-1)})\delta_{x^{(2l+1)}_{2(i-1)}}+\bar n(x^{(2l+1)}_{2k-1}+h)\delta_{x^{(2l+1)}_{2k-1}+h}+\sum_{i=j+1}^{l+1}\bar n(x^{(2l+1)}_{2i-1})\delta_{x^{(2l+1)}_{2i-1}}$\\
        \vspace{4mm}
        if there exists $1\leq j\leq l+1$ s.t.
        $x^{(2l+1)}_{2j-1}\prec x^{(2l+1)}_{2k-1}+h\prec x^{(2l+1)}_{2j}$,\\
    \item  $\Gamma^{(2l+1)}=\sum_{i=1}^j\bar n(x^{(2l+1)}_{2(i-1)})\delta_{x^{(2l+1)}_{2(i-1)}}+\sum_{i=j}^{l+1}\bar n(x^{(2l+1)}_{2i-1})\delta_{x^{(2l+1)}_{2i-1}}$\\ \vspace{4mm}
        if there exists $1\leq j\leq l+1$ s.t.
        $x^{(2l+1)}_{2j-2}\prec x^{(2l+1)}_{2k-1}+h\prec x^{(2l+1)}_{2j-1}$.
    \end{itemize}
    Then, we relabel the traits according to the total order relation as in (C1):
    \beq
    x_0^{(2l+2)}\prec x_1^{(2l+2)}\prec\cdots\prec  x_{2l+1}^{(2l+2)}\prec x_{2l+2}^{(2l+2)},
    \eeq
    where in associate with the first case
    \begin{align*}
    &x_i^{(2l+2)}:=x_i^{(2l+1)} ~\textrm{for} ~0\leq i\leq 2j-1,\quad x^{(2l+2)}_{2j}:=x^{(2l+1)}_{2k-1}+h,\\
    &x_i^{(2l+2)}:= x_{i-1}^{(2l+1)}~\textrm{for} ~2j+1\leq i\leq 2l+2,
    \end{align*}
    and in associate with the second case
    \begin{align*}
    &x_i^{(2l+2)}:=x_i^{(2l+1)} ~\textrm{for} ~0\leq i\leq 2j-2,\quad x^{(2l+2)}_{2j-1}:=x^{(2l+1)}_{2k-1}+h,\\
    &x_i^{(2l+2)}:= x_{i-1}^{(2l+1)}~\textrm{for} ~2j\leq i\leq 2l+2.
    \end{align*}
\end{description}

\end{Def}
\begin{rem}
 According to the definition, the new configuration is constructed in a way that every alternative trait gets
stabilized when one ``looks down''  from the most fittest trait along the declining fitness landscape. Once a mutant is inserted in between two levels (say, $i$ and $i+1$), we relabel all the traits above from the mutant's level. However, the mutation only alters the equilibrium configuration below $i+1$th level but not above. This construction is similar to the look-down construction of Coalescent processes (see \cite{DK99}).  
\end{rem}

\begin{teor}\label{TST_Theorem_infinite trait_Micro}
Admit assumption (\textbf{A1}) and (\textbf{C}). Consider the process $\{X^{K,\epsilon,\sigma}_t, t\geq 0\}$ described by the generator \eqref{generator_X_epsilon,sigma_Micro_nearest neighor}. Suppose that $X_0^{K,\epsilon,\sigma}= \frac{N^K_0}{K}\delta_{x_0}$ and $\frac{N^K_0}{K}\to \bar n(x_0)$ in law as $K\to\infty$. In addition to the condition \eqref{rare migration condition}, suppose it also holds that
\beq\label{rare mutation condition}
\ln\frac{1}{\epsilon}\ll\frac{1}{K\sigma}\ll e^{KC} \qquad\text{for any} ~C>0.
\eeq
Then $(X^{K,\epsilon,\sigma}_{t/K\sigma})_{t\geq 0}$ converges as $K\to \infty$ to the trait substitution tree $(\Gamma_{t})_{t\geq 0}$ defined in Definition \ref{TST_Definition_infinite tree_micro} in the sense of f.d.d. on $\M_F(\X)$ equipped with the topology induced by mappings $\nu\mapsto\langle\nu,f\rangle$ with $f$ a bounded measurable function on $\X$.
\end{teor}

 \begin{figure}[hbtp]
 \centering
\includegraphics[width=220pt]{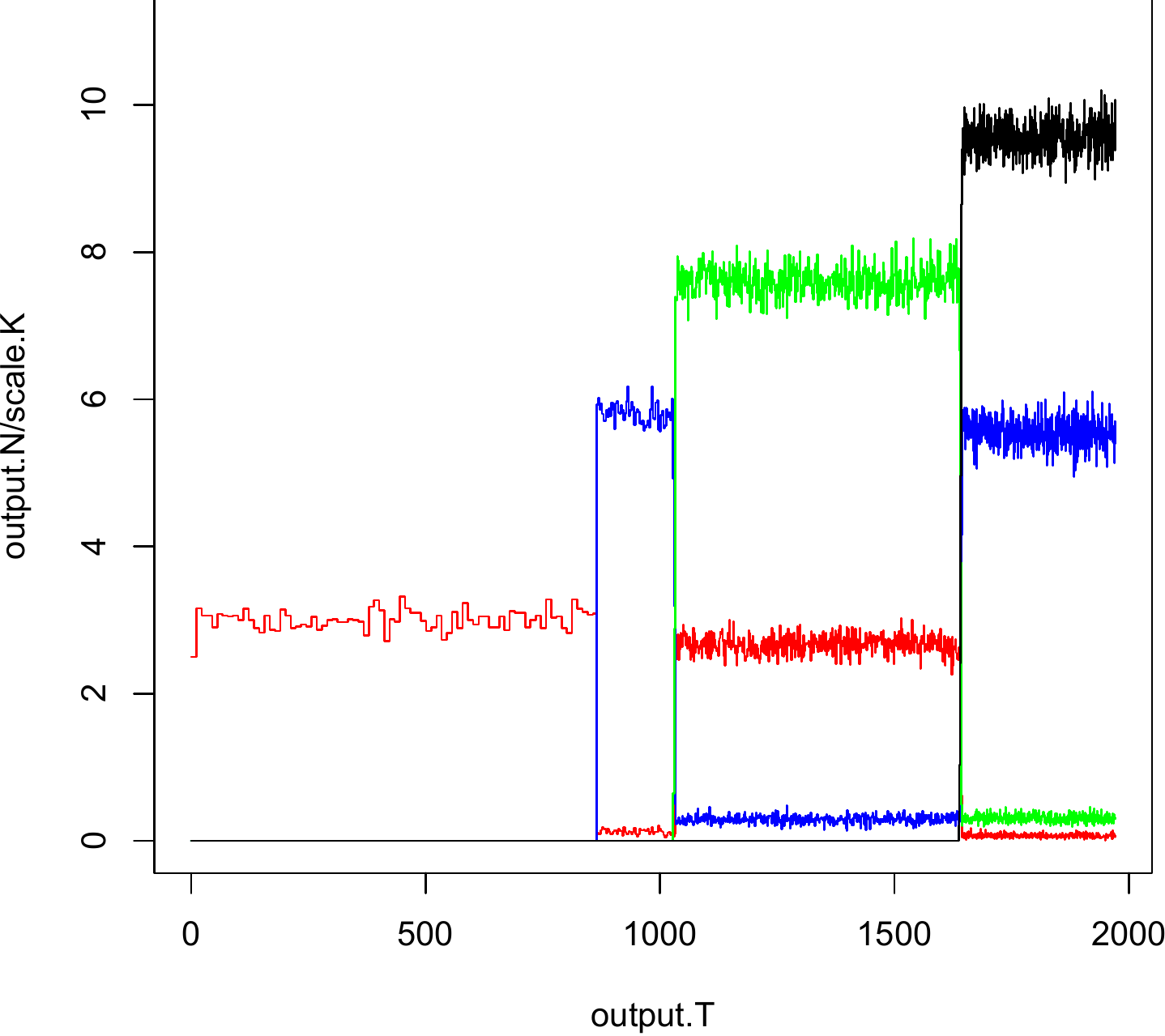}
\includegraphics[width=220pt]{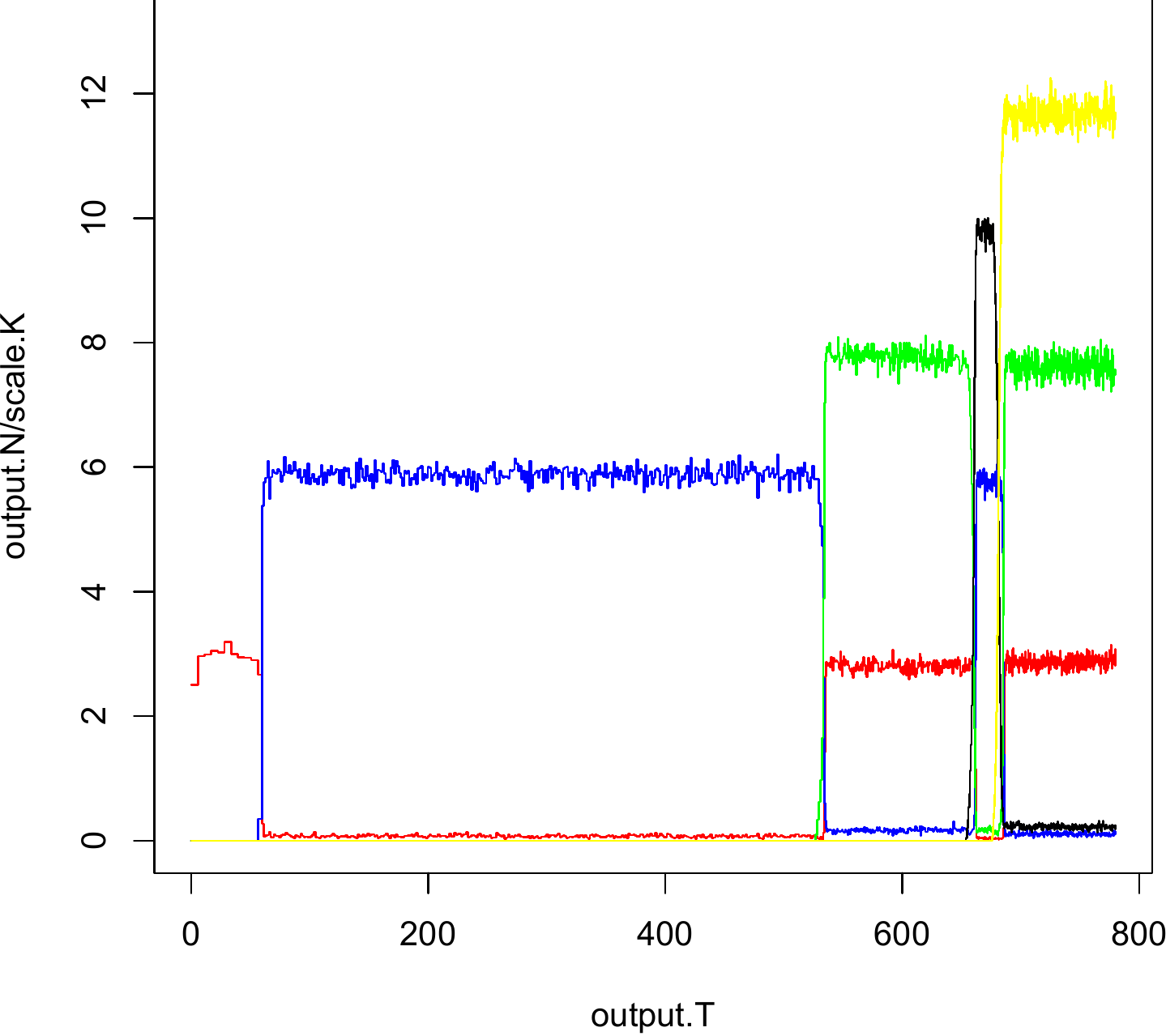}
\caption{\footnotesize Simulations of a trait substitution tree on the mutation time scale arising in Theorem \ref{TST_Theorem_infinite trait_Micro} on four- and five-type trait space. }
\label{Figure TST_simulation_Micro with mutation}
 \end{figure}

\begin{rem}
\begin{itemize}
\item
There are two time scales for the individual-based population, which can be observed from  Theorem \ref{TST Theorem on finite trait space_Micro} and the generator \eqref{generator_X_epsilon,sigma_Micro_nearest neighor}. One is the fixation time scale of order $\ln\frac{1}{\epsilon}$ while the other one is the mutation time scale of order $\frac{1}{K\sigma}$, which are constrained on LHS of the inequality \eqref{rare mutation condition}. By adopting the time scales separation technique used in \cite{Cha06}, we can get a nice limiting structure-TST in the large population limit. The RHS of the inequality \eqref{rare mutation condition} is used to guarantee that system can not drift away from the TST equilibrium configuration on the mutation time scale (see Freidlin and Wentzell \cite{FW84}).

\item
 As it is showed in Figure \ref{Figure TST_simulation_Micro with mutation}, we simulate the trait substitution tree processes by introducing a mutation mechanism. Note that the simulation shows a special case where the population always reproduces a mutant which is more fitter than any of already existing traits. The birth rate of red-colored population is $3$, while the blue one, the green one, the black one and the yellow one have birth rates $6,\,8,\,10,\,12$ resp.. Their death rates are constant $0$. We take $\epsilon=K^{-0.8}$ and $\sigma=K^{-1.5}$, where initial scaling parameter $K=400$.
 On a longer mutation time scale, the fixation process due to migration is not visible any more. However, if we zoom into the infinitesimal fixation period, pictures as in Figure \ref{Figure TST_simulation_Micro} will emerge.

\end{itemize}
\end{rem}


\section{Application to a diploid population}\label{section 3.5}
We begin by introducing some terminology from population genetics in
the same setting as \cite[Chapter 10.1]{EK86}.

In this section we restrict our attention on one-locus diploid populations,
where the chromosomes occur in the form of homologous pairs. More
precisely, an individual's genetic makeup with respect to a
particular locus, as indicated by the unordered pair of alleles
situated there (one on each chromosome), is referred to be its
genotype. Therefore, if there are $n$ possible alleles, i.e. finite allele space
$\mathcal{A}^{n}:=\{A_1,\ldots, A_n\}$, at a given locus, they can fuse
$h(n):=\left(\begin{array}{c} n\\2\end{array}\right)+n=n(n + 1)/2$ possible genotypes. We denote the
genotype space by $G^n:=\{(A_i, A_j), 1\leq i\leq j\leq n\}$.
Without loss of generality one can take the total allele space
$\mathcal{A}$ to be a compact subset of $\R^d$, or simply a continuum
interval $[0,1]$. Denote the total genotype space by $G$. To keep
consistent with notations in Section \ref{section 3.2}, we can endow
quantitative trait value for every different genotype by the
following measurable mapping $\Phi: G^n\mapsto \X$ such that
\beq\label{mapping gene and trait}
\Phi((A_i, A_j))=x_l \qquad\textrm{for}~ 1\leq l\leq h(n), ~1\leq
i\leq j\leq n
\eeq
where $\Phi$ is symmetric for $i,\,j$, and traits
$(x_l)_{1\leq l\leq h(n)}$ are ordered according to their relative
fitness as in assumption (C1). In accordance with notations in
Section \ref{section 3.2}, we list the following parameters for
diploid populations. For any $g\in G$,
\begin{itemize}
  \item denote by $B(g):=b(\Phi(g))$ the birth rate from an individual of genotype $g$.
  \item denote by $D(g):=d(\Phi(g))$ the death rate of an individual of genotype $g$ because of ``aging''.
  \item denote by $\hat\alpha (g_1, g_2):=\alpha(\Phi(g_1), \Phi(g_2))$ the competition kernel felt by an individual of genotype $g_1$ from another of genotype $g_2$.
  We restrict the competition acting on individuals of the same genotype and within the nearest-neighbors (ordered in terms of their fitness values).
  \item denote by $\hat m(g_1, dg_2)$ the replacement law of an individual of genotype $g_1$ by another $g_2$
  when one of allele pair of the mother individual $g_1$ undergoes fusion with another gamete allele from any father to form a new type
  $g_2$. Thus, $g_2$ is chosen to be any genotype fused in a way that one allele is from $g_1$ and the other one can be anyone from the whole supporting allele space.
  \item denote by $\hat\mu(g):=\mu(\Phi(g))$ the mutation rate of an individual of genotype $g$.
  \item denote by $\hat p(A,da)$ the law of mutant variation between a mutant and its resident allele type $A$.
\end{itemize}

We also limit our discussion to monoecious populations, those in
which each individual can act as either a male or a female parent.
The reproductive process can be briefly described as follows.

\textbf{Sexual reproduction}. Sexual reproduction is a creation of a new organism 
by combining genetic gametes from two parental individuals.
Suppose each individual can produce a large amount
of germ cells, cells of the same genotype. Some germ cells split
into two gametes (each carries one chromosome information from the
homologous pair in the original cell), and any two gametes from
different individuals form another genotype which is different from
the parents. The above procedure is called meiosis and fusion of gametes. Note that
sexual reproduction can be realized in a manner of gamete replacement of one of the mother individual's gamete. 
For instance, the effect of reprodcution from a (mother) individual of genotype
$(A_i, A_j)$  by fusing with any gamete $A_k$ generated from any
other (father) individual is a replacement of the individual $(A_i,
A_j)$ by a new individual of genotype either $(A_i, A_k)$ or $(A_k, A_j)$.
We assume that each father individual can generate enough amount of
germ cells to provide gametes for the replacement procedure of
mother individuals. The size of offspring individuals resulted from
sexual reproduction only depends on the size of replaced mother
individuals. Based on the above simplicity, one can
think of sexual reproduction as a spatial migration of an individual from a genotypic
site $(A_i, A_j)$ (say $g_1$) to another genotypic site $(A_i, A_k)$
or $(A_k, A_j)$ (say $g_2$) with migration kernel $\hat m(g_1, dg_2)$,
and the migration rate only depends on the size density of the
(mother) genotype $(A_i, A_j)$.
 Then we can employ the results on the spatial migration model in
Section \ref{section 3.3} and Section \ref{section 3.4}, which
follows later. Here the sexual reproduction is based on allelic level
but is represented on genotypic level. 

\textbf{Clonal birth and death}. Besides the sexual reproduction,
meanwhile, there are also reproductive birth which does not apply
fusion of gametes from two parents. Instead, an offspring carries a clonal copy of the
parent's genotype $g$ with birth rate $B(g)$. This kind of local
birth is carried out as defined in Section \ref{section 3.2}.
Similarly as before, the death of an individual is governed by a
quadratic form of its own size density and the density of its
nearest-neighbors. These two events are based on genotypic level.

\textbf{Mutation}. Mutation occurs due to a change of one of alleles
in a genotype, and the allele space is enlarged by one. Suppose
there are $n$ different allele types before a mutation event. After mutation, the genotype space is enlarged to
be of size $h(n+1)$ from $h(n)$. The mutation event is governed by a mutation
kernel $\hat p(A, da)$ on the allele space $\mathcal{A}$. The mutation event is based on allelic level. 

Notice that all the above events are density dependent, which means
their transition rates are proportional to the local density of the
population size. We consider it as finite measure-valued processes
$(\mathfrak{g}_t)_{t\geq 0}$ on the genotype space $G$. Denote by
$\M_F(G)$ the totality of finite measures on $G$. Now we write down
the infinitesimal generator of the diploid model with $1/K$-scaled
weight for a proper testing function $F$

\beq\bea
 \mathcal{L}^K F(\mathfrak{g})
    &= \int_{G}\left[F(\mathfrak{g}+\frac{\delta_g}{K})-F(\mathfrak{g})\right]B(g)K\mathfrak{g}(dg)\\
    &+\int_{G}\left[F(\g-\frac{\delta_g}{K})-F(\g)\right]\left(D(g)+\int_{G}\hat\alpha(g,\tilde g)\g(d\tilde g)\right){K}\g(dg)\\
    &+\epsilon\int_{G}\int_{G}\left[F(\g+\frac{\de_{g_2}}{K}-\frac{\de_{g_1}}{K})-F(\g)\right]\hat m(g_1,dg_2)1_{\{g_2: g_1\cap g_2\neq \varnothing
     ; g_2\in\textrm{supp}\{\mathfrak{g}\}\}}K\g(dg_1)\\
    &+\sigma\int_{G}\int_{\mathcal{A}}\left[F(\g+\frac{\de_{(A_1+a,A_2)}}{K})-F(\g)\right]\de_{g}((A_1,A_2))\hat\mu(g)\hat p(A_1,da)K\g(dg)
\eea \eeq

where $g_1\cap g_2\neq \varnothing$ means that genotype $g_1$ and
$g_2$ share at least one same allelic type in their allele pairs.

\begin{Def} A Markov jump process $(\Delta_t)_{t\geq 0}$ with initial
$\Delta_0=n_0\delta_{(A_1,A_1)}$ for some $A_1\in\mathcal{A}$ is called a genotype substitution
tree if it satisfies the follows. For any given $n\in\N$ and
$\mathcal{A}^n=\{A_1,A_2,\ldots,A_n\}$, let
$\Gamma^{(h(n))}$ (defined in Definition \ref{TST_Definition_infinite tree_micro}) 
be the equilibrium configuration for
traits distribution determined by genotypes in $G^n=\{(A_i, A_j): 1\leq
i\leq j\leq n\}$ and \eqref{mapping gene and trait}. Let $\Delta^{(n)}$ be $\Gamma^{(h(n))}$'s equivalent form
as the corresponding equilibrium on the genotype space.
Suppose that the number $h(n)$ is even, the
transition rate for the Markov jump process $(\Delta_.)$ from $\Delta^{(n)}$
to $\Delta^{(n+1)}$ due to mutation variation of magnitude $a$ on an
allele $A_j$ $(1\leq j\leq n)$ is
\beq
\sum\limits_{i=1}^{h(n)/2}\bar
n(x_{2i})\hat\mu(x_{2i})(1+\delta_j(k))1_{\{x_{2i}=\Phi((A_j, A_k)), \,\exists\,
1\leq k\leq n\}}\hat p(A_j, da)
\eeq 
where $(x_i)_{1\leq i\leq h(n)}$ are
indexed according to a fitness increasing order of $\{\Phi((A_i, A_j)):
1\leq i\leq j\leq n\}$.
\end{Def}

Note that $(1+\delta_j(k))$ appeared in above rate is used to distinguish cases
whether the mutant allele is from a homozygous genotype or not. 

\begin{teor}
Consider a sequence of processes $\{(\g^K_t)_{t\geq 0}, K\in\N\}$
defined by above generator with initial condition
$\g^K_0=n_K\de_{(A_1,A_1)}$, where $n_K$ converges to $n_0$ as
$K\to\infty$. If mutation and migration rates satisfy constraints
\eqref{rare migration condition} and \eqref{rare mutation
condition}, as $K\to \infty$,
$\left(\g^K_{\frac{t}{K\sigma}}\right)$ converges to a genotype
substitution tree process $(\Delta_t)$ in the sense of f.d.d. on
$\M_F(G)$ equipped with the topology induced by mappings
$\g\mapsto\langle\g,f(\Phi(\cdot))\rangle$ with $f$ a bounded
measurable function on $\X$.
\end{teor}
The above theorem can be obtained as a corollary of Theorem \ref{TST_Theorem_infinite trait_Micro}.


\section{Outline of proofs}\label{section 3.6}

 In order to illustrate the basis idea of proofs, we start with a three-trait toy model.
But notice that our analysis is not reduced only to the three-trait case. All the machinery is still available for any finite-trait space, 
which will be shown later. However, the explicit proofs are more difficult to write down without some restrictive conditions. 
That is why we impose assumption (\textbf{B}3) in Theorem \ref{TST Theorem on finite trait space_Micro}.

\begin{prop}\label{Proposition_3type_Micro}
Admit the same condition as in Theorem \ref{TST Theorem on finite trait space_Micro}.
Consider a sequence of processes on a trait space $\X=\{x_0, x_1, x_2\}$. Then, there exists a constant $\bar{t}_2>0$, such that for any $t>\bar{t}_2$
\beq
\lim\limits_{K\to\infty} X^K_{t\ln\frac{1}{\epsilon}}\stackrel{\text{d}}{=}\Gamma^{(2)}
\eeq
under the total variation norm.
\end{prop}

\begin{proof}(see Figure \ref{Figure_Phase3type_Micro}).

Let $\xi_t^K(x_0):=\frac{N_t^K}{K}$ and $\xi_t^K(x_i):=\frac{N_t^{K,i}}{K}=\langle X_t^K, 1_{\{x_i\}}\rangle$ for $i=1,2$.

 \begin{figure}[hbtp]
 \centering
 \def\svgwidth{400pt}
 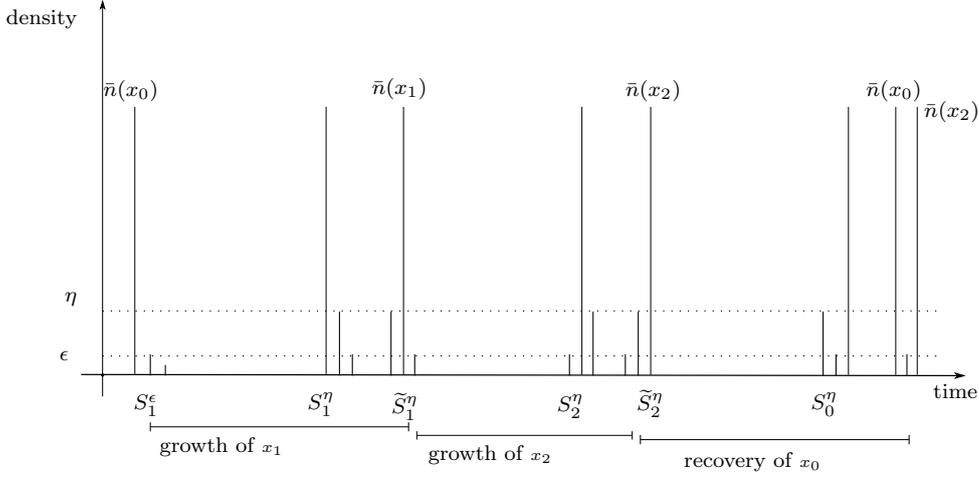
 \caption{{\footnotesize Phase evolution of mass bars in early time window on the three-trait site space}}.
 \label{Figure_Phase3type_Micro}
 \end{figure}

\textbf{Step 1}. Firstly, consider the emergence and growth of population at trait site $x_1$. Set $S_1^{\epsilon}=\inf\{t>0: \xi_t^K(x_1)\geq \epsilon\}$. Thanks to $\frac{N_0^K}{K}\to n_0>0$ in law as $K\to\infty$ and by applying the law of large numbers of random processes (see Chap.11, Ethier and Kurtz 1986), one obtains from the last term in generator \eqref{generator_X_epsilon} that, for any $\de>0,\,T>0$,
\[
\lim\limits_{K\to\infty}\PP\left(\sup\limits_{0\leq t\leq T}\left|\frac{\xi_t^K(x_1)}{\epsilon}-n_t(x_1)\right|<\de\right)=1
\]
where $n_t(x_1)$ is governed by equation $\dot{n}(x_1)= m(x_0,x_1)n_0$ with initial $n_0(x_1)=0$. Therefore,
\beq\label{linear growth before epsilon threshold}
\lim\limits_{K\to\infty}\PP\left(\frac{1}{m(x_0,x_1)n_0}-\de<S_1^{\epsilon}<\frac{1}{m(x_0,x_1)n_0}+\de\right)=1,
\eeq
that is, $S_1^{\epsilon}$ is of order 1.

For any $\eta>0$, set $S_1^{\eta}=\inf\{t: t>S_1^{\epsilon},\,\xi_t^K(x_1)\geq \eta\}$. Consider a sequence of rescaled processes $\left(\frac{N_t^{K,1}}{K\epsilon}\right)_{t\geq S_1^{\epsilon}}$ with $\frac{N_{S_1^{\epsilon}}^{K,1}}{K\epsilon}=\frac{\xi_{S_1^{\epsilon}}^K(x_1)}{\epsilon}\to 1$ as $K\to\infty$. As before, by law of large numbers of random processes (see Chap.11 Ethier and Kurtz 1986), one obtains, for any $\de>0, \,T>0$,
\beq\label{lln before eta threshold}
\lim\limits_{K\to\infty}\PP\left(\sup\limits_{0\leq t\leq T}\left|\frac{N_t^{K,1}}{K\epsilon}-m_t\right|<\de\right)=1,
\eeq
where $m_t$ is governed by equation $\dot{m}=\bar{f}(x_1,x_0)m=\left(b(x_1)-d(x_1)-\al(x_1,x_0)\bar{n}(x_0)\right)m$ with $m_0=1$.

Set $T_1^{\eta/\epsilon}=\inf\{t-S_1^{\epsilon}: t>S_1^{\epsilon},\,\frac{N_t^{K,1}}{K\epsilon}\geq \eta/\epsilon\}$, and $t_1^{\eta/\epsilon}=\inf\{t>0: m_t\geq \eta/\epsilon\}$.
Then, for any $\de>0$, there exists $\de^{'}>0$ such that
\beq\label{time estimate before eta threshold}
\bea
\lim\limits_{K\to\infty}&\PP\left(\left(\frac{1}{\bar{f}(x_1,x_0)}-\de\right)\ln\frac{1}{\epsilon}< S_1^\eta-S_1^{\epsilon}<\left(\frac{1}{\bar{f}(x_1,x_0)}+\de\right)\ln\frac{1}{\epsilon}\right)\\
=\lim\limits_{K\to\infty}&\PP\left(\left(\frac{1}{\bar{f}(x_1,x_0)}-\de\right)\ln\frac{1}{\epsilon}< T_1^{\eta/\epsilon}<\left(\frac{1}{\bar{f}(x_1,x_0)}+\de\right)\ln\frac{1}{\epsilon}\right)\\
=\lim\limits_{K\to\infty}&\PP\Bigg(\left(\frac{1}{\bar{f}(x_1,x_0)}-\frac{\de}{2}\right)\ln\frac{1}{\epsilon}< t_1^{\eta/\epsilon}<\left(\frac{1}{\bar{f}(x_1,x_0)}+\frac{\de}{2}\right)\ln\frac{1}{\epsilon},\\&\sup\limits_{0\leq t\leq t_1^{\eta/\epsilon}}\mid \frac{N_t^{K,1}}{K\epsilon}-m_t\mid<\de^{'}\Bigg)\\
=&1
\eea
\eeq
where the last equal sign is due to \eqref{lln before eta threshold}.

After population of trait $x_1$ reaches some $\eta$ threshold, the dynamics $\left(\xi_t^K(x_0),\xi_t^K(x_1)\right)$ can be approximated by the solution of a two-dimensional Lotka-Volterra equations. Then, it takes time of order 1 (mark this time coordinator by $\widetilde S_1^{\eta}$) for the two subpopulations switching their mass distribution and gets attracted into $\eta-$neighborhood of the stable equilibrium $\left(0,\bar{n}(x_1)\right)$.

\textbf{Step 2.} Now consider the emerging and growth of population $\xi_t^K(x_2):=\langle X_t^K, 1_{\{x_2\}}\rangle$ at trait site $x_2$. Set $S_2^{\epsilon}=\inf\{t: t>\widetilde{S}_1^{\eta}, \,\xi_t^K(x_2)\geq \epsilon\}$.
Similarly as is done for $S_1^{\epsilon}$ in \eqref{linear growth before epsilon threshold}, one can get that $\lim\limits_{K\to\infty}\PP(S_2^{\epsilon}-\widetilde{S}_1^{\eta}=O(1))=1$. On a longer time scale, we will not distinguish $S_2^{\epsilon}$ from $\widetilde{S}_1^{\eta}$.

Set $S_2^{\eta}=\inf\{t: t>S_2^{\epsilon},\, \xi_t^K(x_2)\geq \eta\}$. One follows the same procedure to derive \eqref{time estimate before eta threshold} and asserts that for any $\de>0$,
\beq\label{time estimate before second eta threshold}
\lim\limits_{K\to\infty}\PP\left(\left(\frac{1}{\bar{f}(x_2,x_1)}-\de\right)\ln\frac{1}{\epsilon}< S_2^\eta-\widetilde S_1^{\eta}<(\frac{1}{\bar{f}(x_2,x_1)}+\de)\ln\frac{1}{\epsilon}\right)
=1.
\eeq
Note that assumption (\textbf{B}3) $\frac{2}{b(x_2)-d(x_2)}\geq \frac{1}{\bar f(x_1,x_0)}+\frac{1}{\bar f(x_2,x_1)}$ guarantees that $\xi^K_t(x_2)$ can not grow so fast in exponential rate $b(x_2)-d(x_2)$ such that it reaches some $\eta$-level before $S_2^{\eta}$.

During time period $(\widetilde S_1^{\eta}, S_2^{\eta})$, population at site $x_0$, on one hand, decreases due to the competition from more fitter trait $x_1$. On the other hand, it can not go below $\epsilon$ level due to the successive migration in a portion of $\epsilon$ from site $x_1$. More precisely, by neglecting migrant contribution,
$\xi_t^K(x_0)$ converges $n_t(x_0)$ in probability as $K$ tends to $\infty$, where
\beq
\dot{n}_t(x_0)=\left(b(x_0)-d(x_0)-\al(x_0,x_1)\bar{n}(x_1)\right) n_t(x_0)=\bar{f}(x_0,x_1)n_t(x_0)
\eeq
with $n_0(x_0)=\eta$. Let $\Delta S_2^\eta=S_2^\eta-\widetilde S_1^{\eta}$. Then, for any $\de>0$,
\beq
\bea
&\lim\limits_{K\to\infty}\PP\left(\xi_{S_2^{\eta}}^K(x_0)\in(n_{\Delta
  S_2^{\eta}}(x_0)-\de,n_{\Delta S_2^{\eta}}(x_0)+\de)\right)\\
&=\lim\limits_{K\to\infty}\PP\left(\eta
  e^{\bar{f}(x_0,x_1)\Delta S_2^{\eta}}-\de<\xi_{S_2^{\eta}}^K(x_0)<\eta e^{\bar{f}(x_0,x_1)\Delta S_2^{\eta}}+\de\right)\\
&=\lim\limits_{K\to\infty}\PP\left(\eta
  \epsilon^{\mid\bar{f}(x_0,x_1)\mid/\bar{f}(x_2,x_1)}-\de<\xi_{S_2^{\eta}}^K(x_0)<\eta \epsilon^{\mid\bar{f}(x_0,x_1)\mid/\bar{f}(x_2,x_1)}+\de\right)\\
&=1
\eea
\eeq
where the second equality is due to \eqref{time estimate before second eta threshold}.
Taking the migration from site $x_1$ into account, we thus have
\beq\label{density estimate original type after second eta threshold}
\lim\limits_{K\to\infty}\PP\left(\xi_{S_2^{\eta}}^K(x_0)=O(\epsilon^{\mid\bar{f}(x_0,x_1)\mid/\bar{f}(x_2,x_1)}
\vee \epsilon)\right)=1.
\eeq
We proceed as before for $\widetilde S_1^{\eta}$ in step 1. After time $S_2^{\eta}$, the mass bars on dimorphic system $(\xi_t^K(x_1),\xi_t^K(x_2))$ can be approximated by ODEs and will be switched again in time of order 1 (marked by $\widetilde S_2^{\eta}$ as in Figure \ref{Figure_Phase3type_Micro}), and they are attracted into $\eta-$ neighborhood of $(0,\bar{n}(x_2))$. As for the population density on site $x_0$, one obtains from \eqref{density estimate original type after second eta threshold}
\beq
\lim\limits_{K\to\infty}\left(\xi^K_{\widetilde{S}^{\eta}_2}(x_0)=O(\epsilon^{c_1})\right)=1
\eeq
where $c_1=\frac{|\bar f(x_0,x_1)|}{\bar f(x_2,x_1)}\wedge 1\leq 1$.

\textbf{Step 3.} We now consider the recovery of subpopulation at trait site $x_0$. Recovery arises because of the lack of effective competitions from its neighbor site $x_1$, or under negligible competitions since the local population density on $x_1$ is very low under the control of its fitter neighbor $x_2$.
Without lose of generality, we suppose $c_1:=\frac{\mid\bar{f}(x_0,x_1)\mid}{\bar f(x_2,x_1)}<1$ in \eqref{density estimate original type after second eta threshold}.

Set $S_0^{\eta}=\inf\{t: t>\widetilde S_2^{\eta}, \,\xi_t^K(x_0)\geq \eta\}$. We proceed as before in step 1. From \eqref{density estimate original type after second eta threshold}, $\frac{\xi_{\widetilde S_2^{\eta}}^K(x_0)}{\epsilon^{c_1}}$ converges to some positive constant (say $m_0$) in probability as $K\to\infty$. Thus, by applying law of large numbers to the sequence of processes $\frac{N^K_t}{K\epsilon^{c_1}}$, for any $\de>0,\,T>0,$
\beq
\lim\limits_{K\to\infty}\PP\left(\sup\limits_{0\leq t\leq T}\left|\frac{\xi_t^K(x_0)}{\epsilon^{c_1}}-m_t\right|<\de\right)=1
\eeq
where $m_t$ is governed by logistic equation $\dot{m}=\left(b(x_0)-d(x_0)\right)m$ starting with a positive initial $m_0$.

Following the same way to obtain \eqref{time estimate before eta threshold}, time length $S_0^{\eta}-\widetilde S_2^{\eta}$ can be approximated by time needed for dynamics $m$ to approach $\eta/\epsilon^{c_1}$ level, which is of order $\frac{c_1}{(b(x_0)-d(x_0))}\ln\frac{1}{\epsilon}$, i.e. for any $\de>0$,
\beq\label{time estimate for original recovery}
\lim\limits_{K\to\infty}\PP\left(\left(\frac{c_1}{b(x_0)-d(x_0)}-\de\right)\ln\frac{1}{\epsilon}< S_0^\eta-\widetilde S_2^{\eta}<\left(\frac{c_1}{b(x_0)-d(x_0)}+\de\right)\ln\frac{1}{\epsilon}\right)
=1.
\eeq
At the same time, $\xi^K_t(x_1)$ converges in probability to $\psi_t$ which satisfies equation $\dot{\psi}=\bar{f}(x_1,x_2)\psi$ with $\psi_{\widetilde{S}^{\eta}_2}=\eta$. Then, we can justify the following estimate for population density at site $x_1$,
\beq\label{density estimate first type after original recovery}
\lim\limits_{K\to\infty}\PP\left(\xi_{S_0^{\eta}}^K(x_1)=O(\epsilon^{c_2}
\vee \epsilon)\right)=1
\eeq
where $c_2=\frac{c_1\mid\bar{f}(x_1,x_2)\mid}{b(x_0)-d(x_0)}$.

We now combine all these estimates \eqref{time estimate before eta threshold}, \eqref{time estimate before second eta threshold}, \eqref{time estimate for original recovery} together, and conclude that
\beq
\lim\limits_{K\to\infty}\PP\left(\| X^K_{t\ln\frac{1}{\epsilon}}-\Gamma^{(2)}\|<\de\right)=1
\eeq
for $t>\bar{t}_2:=\frac{1}{\bar{f}(x_1,x_0)}+\frac{1}{\bar{f}(x_2,x_1)}+\frac{c_1}{b(x_0)-d(x_0)}$ under the total variation norm $\|\cdot\|$ on $\M_F(\X)$.
\end{proof}

\begin{proof}[Proof of Theorem \ref{TST Theorem on finite trait space_Micro}]
We proceed the proof by the induction method over the superscript $L\in\N$ of trait space $\X^{(L)}=\{x_0, x_1,\ldots,x_L\}$.

(1). When $L=2$, it is already proved in Proposition \ref{Proposition_3type_Micro} that there exists a constant $\bar t_2>0$ such that for any $t>\bar t_2$
\beq
\lim\limits_{K\to\infty} X^K_{t\ln\frac{1}{\epsilon}}\stackrel{\text{(d)}}{=}\Gamma^{(2)}
\eeq
under the total variation norm.

(2). Without loss of generality,  suppose it holds that for any $L=2l$ there exists a constant $\bar t_{2l}$ such that for any $t>\bar t_{2l}$
\beq
\lim\limits_{K\to\infty} X^K_{t\ln\frac{1}{\epsilon}}\stackrel{\text{(d)}}{=}\Gamma^{(L)}.
\eeq
We need to prove the same relation also holds for the case $L=2l+1$.

We firstly consider the invasion time scale of population at site $x_{2l+1}$.

Denote by $\xi^K_t(x_{2l+1}):=\langle X^K_t, 1_{\{x_{2l+1}\}}\rangle$. If $K\epsilon^{2l+1}\ll 1$, it follows a similar proof as in Proposition \ref{Proposition_3type_Micro}. So, now we only need to consider the case when $K\epsilon^{2l+1}\gg 1$, that is, the mass at site $x_{2l+1}$ is large in the very beginning. In fact, since $\frac{N_0^K}{K}\to n_0$ in law as $K\to\infty$ and the nearest-neighbor mass migrates from site $x_0$ to site $x_{2l+1}$ by passing through $x_1,\ldots, x_{2l}$, one applies the law of large numbers for random processes and obtains that
\beq
\lim\limits_{K\to\infty}\PP\left(\sup\limits_{0\leq t\leq T}\left|\xi^K_t(x_{2l+1})-n_t(x_{2l+1})\right|<\de\right)=1
\eeq
where $n_t(x_{2l+1})$ satisfies the equation $\dot{n}_t(x_{2l+1})=\epsilon^{2l+1}\prod_{j=1}^{2l+1}m(x_j,x_{j-1})n_0$.
So, it takes time of order 1 for $\xi^K_t(x_{2l+1})$ to reach $\epsilon^{2l+1}$ level (mark the time coordinator by $S^{\epsilon}_{2l+1}$).

Set $S^{\eta}_{2l+1}=\inf\{t: t> S^{\epsilon}_{2l+1}, \xi_t^K(x_{2l+1})\geq \eta\}$. For $t\in\left(S^{\epsilon}_{2l+1},S^{\eta}_{2l+1}\right)$, again by law of large numbers, $\frac{\xi^K_t(x_{2l+1})}{\epsilon^{2l+1}}$ converges to $\phi_t$ which satisfies $\phi_0=1$ and
\beq
\dot{\phi}=(b(x_{2l+1})-d(x_{2l+1}))\phi.
\eeq
Thus, $\Delta S^{\eta}_{2l+1}:=S^{\eta}_{2l+1}-S^{\epsilon}_{2l+1}$ can be approximated by the time length (say $\Delta t_{2l+1}$) needed for dynamics $\phi$ to reach $\eta/\epsilon^{2l+1}$ level, i.e.
\beq\label{time estimate for 2l+1 eta threshold1}
\bea
&\lim\limits_{K\to\infty}\PP\left(\left(\frac{2l+1}{b(x_{2l+1})-d(x_{2l+1})}-\de\right)\ln\frac{1}{\epsilon}<\Delta S^{\eta}_{2l+1}<\left(\frac{2l+1}{b(x_{2l+1})-d(x_{2l+1})}+\de\right)\ln\frac{1}{\epsilon}\right)\\
&\lim\limits_{K\to\infty}\PP\left(\left(\frac{2l+1}{b(x_{2l+1})-d(x_{2l+1})}-\de\right)\ln\frac{1}{\epsilon}<\Delta t_{2l+1}<\left(\frac{2l+1}{b(x_{2l+1})-d(x_{2l+1})}+\de\right)\ln\frac{1}{\epsilon}\right)\\
&=1.
\eea
\eeq
We inherit the notation $S^{\eta}_{2l}$ as the hitting time of $\eta$-level for the population at site $x_{2l}$. Due to the hypothesis for $L=2l$ case, we know that $S^{\eta}_{2l}$ is of order
\beq\label{time estimate for 2l eta threshold}
\left(\left[\bar{f}(x_{2l},x_{2l-1})\right]^{-1}+\ldots+\left[\bar{f}(x_1,x_0)\right]^{-1}\right)\ln\frac{1}{\epsilon}.
\eeq
Thanks to assumption (\textbf{B}3), i.e.
\beq
\frac{2l+1}{b(x_{2l+1})-d(x_{2l+1})}>\frac{1}{\bar f(x_{2l},x_{2l-1})}+\ldots+\frac{1}{\bar f(x_1,x_0)},
\eeq
it implies that before time $S^{\eta}_{2l}$, population at site $x_{2l+1}$ is still under negligible level (of order $\epsilon^c$ for some positive constant $c$) and can not influence the invasion process up to $x_{2l}$.

Following a similar procedure as deriving \eqref{time estimate for 2l eta threshold} (see Figure \ref{Figure_Phase3type_Micro}) to analyze the colonization of population at site $x_{2l+1}$ due to migration from site $x_{2l}$ with exponential rate $\bar{f}(x_{2l+1},x_{2l})$, one obtains that $S^{\eta}_{2l+1}$ should be of order
\beq\label{time estimate for 2l+1 eta threshold2}
\left(\left[\bar{f}(x_{2l+1},x_{2l})\right]^{-1}+\ldots+\left[\bar{f}(x_1,x_0)\right]^{-1}\right)\ln\frac{1}{\epsilon}.
\eeq

Comparing two time scale estimates \eqref{time estimate for 2l+1 eta threshold1} and \eqref{time estimate for 2l+1 eta threshold2} for $S^{\eta}_{2l+1}$ under assumption (\textbf{B}3), one gets \eqref{time estimate for 2l+1 eta threshold2} is the right one for the fixation of population at site $x_{2l+1}$.

Now we consider the total recovery time by summing up recovery time of all subpopulation on every second site backwards from $x_{2l+1}$ to $x_0$, one can do calculations repeatedly as in Step 3 of the proof for Proposition \ref{Proposition_3type_Micro}. More precisely, for $1\leq i\leq l$, the initial population $\xi^K(x_{2i-1})$ on site $x_{2i-1}$ which is prepared for recovering is no less than $\epsilon$-level due to the consistent migration from its fitter neighbor site $x_{2i}$. On the other hand, it grows exponentially at least with a rate $\bar{f}(x_{2i-1},x_{2i-2})=b(x_{2i-1})-d(x_{2i-1})-\al(x_{2i-1},x_{2i-2})\bar{n}(x_{2i-2})$ due to the possibly strongest competition from its unfit neighbor $x_{2i-2}$. In all, the recovery time (mark by $S^{\eta,2}_{2i-1}$) for the population $\xi^K(x_{2i-1})$ to reach $\eta$-level can be bounded from above
\beq\label{time estimate for single recovery}
\lim\limits_{K\to\infty}\PP\left(S^{\eta,2}_{2i-1}<\left(\frac{1}{\bar{f}(x_{2i-1},x_{2i-2})}+\de\right)\ln\frac{1}{\epsilon}\right)=1.
\eeq

We now combine both time estimates \eqref{time estimate for 2l+1 eta threshold2} and \eqref{time estimate for single recovery}. Let
\beq
\bar{t}_{2l+1}:=2\left(\left[\bar{f}(x_{2l+1},x_{2l})\right]^{-1}+\ldots+\left[\bar{f}(x_1,x_0)\right]^{-1}\right).
\eeq
Then, one can conclude that for any $t>\bar{t}_{2l+1}$, for any $\de>0$ and $0\leq i\leq l$,

\beq
\bea
&\lim\limits_{K\to\infty}\PP\left(\left|\langle X^K_{t\ln\frac{1}{\epsilon}}, 1_{\{x_{2i+1}\}}\rangle-\bar{n}(x_{2i+1})\right|<\de\right)=1,\\
&\lim\limits_{K\to\infty}\PP\left(\langle X^K_{t\ln\frac{1}{\epsilon}}, 1_{\{x_{2i}\}}\rangle<\de\right)=1.
\eea
\eeq
It follows the conclusion for any $t>\bar{t}_{2l+1}$,
\beq
\lim\limits_{K\to\infty} X^K_{t\ln\frac{1}{\epsilon}}\stackrel{\text{(d)}}{=}\Gamma^{(2l+1)}.
\eeq

\end{proof}

\begin{proof}[Proof of Theorem \ref{TST_Theorem_infinite trait_Micro}]
The proof of this result is similar to the proof of \cite[Theorem 1]{Cha06}. We will not repeat all the details and only focus more on supporting lemmas which are cornerstones of the proof.

For any $\varepsilon >0, \,t>0, \,L\in\N, \,B\subset\X$ measurable, take the integer part $L_1:=\big\lfloor\frac{L+2}{2}\big\rfloor$ and denote by
\beq
\bea
A^{K,\epsilon,\sigma}(\varepsilon,t,L,B): =&\Big\{\textrm{Supp}(X^{K,\epsilon,\sigma}_{\frac{t}{K\sigma}})~\textrm{has}~ L+1 ~\textrm{elements, and}~L_1~\textrm{out of them, ~say}~ \{x_1,\ldots,\\
&\quad x_{L_1}\}\subset B,
 ~\textrm{satisfy}~ \left|\langle X^{K,\epsilon,\sigma}_{\frac{t}{K\sigma}},
 1_{\{x_i\}}\rangle-\bar{n}(x_i)\right|<\varepsilon,\,1\leq i\leq L_1,\\
&\textrm{and the other}~ L+1-L_1 ~\textrm{traits, say}~
 y_1,\ldots,y_{L-L_1}, ~\textrm{satisfy}\\
&\quad\langle X^{K,\epsilon,\sigma}_{\frac{t}{K\sigma}}, 1_{\{y_j\}}\rangle<\varepsilon, 1\leq j\leq L+1-L_1\Big\}.
\eea
\eeq
To the end, it is enough to establish that
\beq
\lim\limits_{K\to\infty}\PP\left(A^{K,\epsilon,\sigma}(\varepsilon,t,L,B)\right)
=\PP\left(\textrm{Supp}(\Gamma_t)\subset B ~\textrm{and has}~L_1~\textrm{elements}~\right)
\eeq
where $(\Gamma_t)_{t\geq 0}$ is defined in Definition \ref{TST_Definition_infinite tree_micro}.

The first key ingredient of the proof is the characterization of exponentially distributed waiting time of each mutation event. 
It can be proved from the expression of the generator \eqref{generator_X_epsilon,sigma_Micro_nearest neighor} 
as done in \cite[Lemma 2 (c)]{Cha06}. We will not show the details here.

\begin{lem}\label{Lemma_exp time in TST_Micro}
Assume that $X_0^{K,\epsilon,\sigma}=\Gamma^{(L)}$, w.o.l., take $L=2l$. Let $\tau$ be the first mutation time after 0. Then,
\beq
\lim\limits_{K\to\infty} \PP\left(\tau>\frac{t}{K\sigma}\right)=\exp\left(-t\sum_{i=0}^l\bar n(x_{2i}^{(2l)})\mu(x_{2i}^{(2l)})\right).
\eeq
\beq
\lim\limits_{K\to\infty} \PP\left(\textrm{at time}\,\tau,\, \textrm{mutant comes from trait}~  x_{2k}^{(2l)}\right)=\frac{\bar n(x_{2k}^{(2l)})\mu(x_{2k}^{(2l)})}{\sum_{i=0}^l\bar n(x_{2i}^{(2l)})\mu(x_{2i}^{(2l)})}.
\eeq
\end{lem}

The second ingredient can been seen as a corollary of Theorem \ref{TST Theorem on finite trait space_Micro}. It demonstrates that fixation of new configuration takes time of order $\ln\frac{1}{\epsilon}$, which is invisible on the mutation time scale.
\begin{lem}\label{Lemma_new config. TST_Micro}
Assume that $X_0^{K,\epsilon,\sigma}=\Gamma^{(2l)}+\frac{1}{K}\delta_{x_{2k}^{(2l)}+h}$ for some $0\leq k\leq l$. Then there exists a constant $C>0$, for any $\delta>0$,  such that
\beq
 \lim\limits_{K\to\infty}\PP\left(\tau>C\ln\frac{1}{\epsilon},\,\sup\limits_{t\in(C\ln\frac{1}{\epsilon}, \tau)}\|X_t^{K,\epsilon, \sigma}-\Gamma^{(2l+1)}\|<\delta\right)=1
\eeq
where $\Gamma^{(2l+1)}$ is defined as in Definition \ref{TST_Theorem_infinite trait_Micro} (i) and $\|\cdot\|$ is the total variation distance.
\end{lem}
\begin{proof}
From Lemma \ref{Lemma_exp time in TST_Micro}, one concludes that, for any $C>0$,
\[
\lim\limits_{\epsilon\to 0}\PP\big(\tau^{\epsilon}>C\ln\frac{1}{\epsilon}\big)=1.
\]
According to the fitness landscape, there will be one and only one ordered position for the new arising trait $x_{2k}^{(2l)}+h$ in $\Gamma^{(2l)}$. Suppose there exists $x_{2j}^{(2l)}$ such that $x_{2k}^{(2l)}+h$ fits between $x_{2j}^{(2l)}$ and $x_{2j+1}^{(2l)}$. Then, one has the local fitness order
\beq
x_{2j-1}^{(2l)}\prec x_{2j}^{(2l)}\prec x_{2k}^{(2l)}+h\prec x_{2j+1}^{(2l)}.
\eeq

Since it is unpopulated for both traits $x_{2j-1}^{(2l)}$ and $x_{2j+1}^{(2l)}$ in $\Gamma^{(2l)}$, we consider $\big(x_{2j}^{(2l)},\, x_{2k}^{(2l)}+h\big)$ as an isolated pair without competition from others. As the same analysis as being done in Proposition \ref{Proposition_3type_Micro}, the two-type system will converge to $\big(0,\,\bar n(x_{2k}^{(2l)}+h)\delta_{x_{2k}^{(2l)}+h}\big)$ in time of order $O\left(\ln\frac{1}{\epsilon}\right)$. On the right hand side of the isolated pair, nothing changes due to their isolation. Whereas on the left hand side of the pair, trait $x_{2j-1}^{(2l)}$ increases exponentially due to the decay of its fitter neighbor $x_{2j}^{(2l)}$. So on and so forth, the mass occupation flips on the left hand side of $x_{2j}^{(2l)}$.  As the same arguments in the finite trait space case (see \emph{Proof of Theorem \ref{TST Theorem on finite trait space_Micro}}), the entire rearrangement process can be completed in time of order $O(\ln\frac{1}{\epsilon})$.

In a similar method, we can prove the other case when the fitness location of $x^{(2l)}_{2k}+h$ is on the left hand side of $x^{(2l)}_{2j}$, that is,

\[
x_{2j-1}^{(2l)}\prec x_{2k}^{(2l)}+h\prec x_{2j}^{(2l)}\prec x_{2j+1}^{(2l)}.
\]

In all, we conclude the new configuration $\Gamma^{(2l+1)}$ by relabeling the traits as done in Definition \ref{TST_Definition_infinite tree_micro} (i).
\end{proof}
Thus we conclude the proof of the Theorem  \ref{TST_Theorem_infinite trait_Micro}.

\end{proof}


\begin{thebibliography}{10}

\bibitem{BP97}
B.~Bolker and S.~Pacala.
\newblock Using moment equations to understand stochastically driven spatial
  pattern formation in ecological systems.
\newblock {\em Theor. Popul. Biol.}, 52:179--197, 1997.

\bibitem{BW11a}
A.~Bovier and S.~D. Wang.
\newblock Trait substitution trees on two timescales analysis.
\newblock 2011.
\newblock Preprint.

\bibitem{Cha06}
N.~Champagnat.
\newblock A microscopic interpretation for adaptive dynamics trait substitution
  sequence models.
\newblock {\em Stoch. Proc. Appl.}, 116:1127--1160, 2006.

\bibitem{CL07}
N.~Champagnat and A.~Lambert.
\newblock Evolution of discrete populations and the canonical diffusion of
  adaptive dynamics.
\newblock {\em Ann. Appl. Probab.}, 17:102--155, 2007.

\bibitem{CM10}
N.~Champagnat and S.~M\'el\'eard.
\newblock Polymorphic evolution sequence and evolutionary branching.
\newblock {\em Probab. Theor. and Relat. Field.}, 148, 2010.

\bibitem{CE09}
A.~Clayton and S.~N. Evans.
\newblock Mutation-selection balance with recombination: convergence to
  equilibrium for polynomial selection costs.
\newblock {\em SIAM J. Appl. Math}, 69:1772--1792, 2009.

\bibitem{CMM11}
P.~Collet, S.~M\'el\'eard, and J.A.J. Metz.
\newblock A rigorous model study of the adaptative dynamics of mendelian
  diploids.
\newblock {\em arXiv:1111.6234v1}, 2011.

\bibitem{DG10}
D.~A. Dawson and A.~Greven.
\newblock Multiscale analysis: Fisher-wright diffusions with rare mutations and
  selection, logistic branching system.
\newblock 2010.

\bibitem{DK99}
P.~J. Donnelly and T.~M. Kurtz.
\newblock A countable representation of the fleming-viot measure-valued
  diffusions.
\newblock {\em Ann. Probab.}, 24:698--742, 1999.

\bibitem{Esh96}
I.~Eshel.
\newblock On the changing concept of evolutionary population stability as a
  reflection of a changing point of view in the quantitative theory of
  evolution.
\newblock {\em J. Math. Biol.}, 34:485--510, 1996.

\bibitem{Eth04}
A.~M. Etheridge.
\newblock Survival and extinction in a locally regulated population.
\newblock {\em Ann. Appl. Probab.}, 14:188--214, 2004.

\bibitem{EK86}
S.~N. Ethier and T.~G. Kurtz.
\newblock {\em Markov proccesses: characterization and convergence}.
\newblock John Wiley and Sons, New York, 1986.

\bibitem{ESW07}
S.~N. Evans, D.~Steinsaltz, and K.~W. Wachter.
\newblock A mutation-selection model for general genotypes with recombination.
\newblock 2007.

\bibitem{FM04}
N.~Fournier and S.~M\'el\'eard.
\newblock A microscopic probabilistic description of a locally regulated
  population and macroscopic approximation.
\newblock {\em Ann. Appl. Probab.}, 14:1880--1919, 2004.

\bibitem{FW84}
M.~I. Freidlin and A.~D. Wentzell.
\newblock {\em Random perturbations of dynamical systems}.
\newblock Springer, New York, 1984.

\bibitem{HW07}
M.~Hutzenthaler and A.~Wakolbinger.
\newblock Ergodic behavioer of locally regulated branching populations.
\newblock {\em Ann. Appl. Probab.}, 17:474--501, 2007.

\bibitem{LD02}
R.~Law and U.~Dieckmann.
\newblock Moment approximations of individual-based models.
\newblock {\em The Geometry of Ecological Interactions: Simplifying Spatial
  Complexity}, pages 252--270, 2002.

\bibitem{MT09}
S.~M\'el\'eard and V.~C. Tran.
\newblock Trait substitution sequence process and canonical equation for
  age-structured populations.
\newblock {\em Journal of Math. Biol}, 58:881--921, 2009.

\bibitem{Wang11}
S.~D. Wang.
\newblock Fixation and substitution of nearly neutral mutants in a locally
  regulated population.
\newblock 2011.
\newblock Preprint.

\end{thebibliography}

\end{document}